\documentclass[11pt,reqno,a4paper]{amsart}

\usepackage[T1]{fontenc}
\usepackage[utf8]{inputenc}

\usepackage{libertinus}
\usepackage{microtype}
\usepackage{geometry}
\usepackage{amsmath,amssymb,amsthm,amsfonts}
\usepackage{mathtools}
\usepackage{mathrsfs}
\usepackage{bbm}

\numberwithin{equation}{section}

\usepackage{graphicx}
\usepackage{xcolor}
\usepackage{enumitem}
\usepackage{comment}
\setlist{itemsep=2pt,topsep=4pt}

\usepackage[
  colorlinks=true,
  linkcolor=blue,
  citecolor=blue,
  urlcolor=blue,
  pagebackref=true
]{hyperref}

\renewcommand*{\backrefalt}[4]{\ifcase #1\relax
  \or
    {\footnotesize Cited on page #2.}\else
    {\footnotesize Cited on pages #2.}\fi
}

\usepackage[nameinlink,noabbrev]{cleveref}
\crefname{section}{Section}{Sections}
\Crefname{section}{Section}{Sections}
\crefname{subsection}{Subsection}{Subsections}
\Crefname{subsection}{Subsection}{Subsections}

\usepackage{titlesec}

\titleformat{\section}
  {\normalfont\large\bfseries\centering}
  {\thesection.}
  {0.75em}
  {}

\titlespacing*{\section}
  {0pt}
  {3.2ex plus 1ex minus .2ex}
  {1.6ex plus .3ex}

\titleformat{\subsection}
  {\normalfont\normalsize\bfseries}
  {\thesubsection.}
  {0.75em}
  {}

\titlespacing*{\subsection}
  {0pt}
  {2.4ex plus .8ex minus .2ex}
  {1.0ex plus .2ex}

\titleformat{\subsubsection}
  {\normalfont\normalsize\bfseries\itshape}
  {\thesubsubsection.}
  {0.75em}
  {}

\titlespacing*{\subsubsection}
  {0pt}
  {2.0ex plus .6ex minus .2ex}
  {0.7ex plus .2ex}

\usepackage{aliascnt}
\usepackage{etoolbox}

\newtheoremstyle{compact}
  {6pt}
  {1pt}
  {\itshape}
  {}
  {\bfseries}
  {.}
  {0.5em}
  {}

\theoremstyle{compact}
\newtheorem{theorem}{Theorem}[section]

\newaliascnt{proposition}{theorem}
\newtheorem{proposition}[proposition]{Proposition}
\aliascntresetthe{proposition}

\newaliascnt{lemma}{theorem}
\newtheorem{lemma}[lemma]{Lemma}
\aliascntresetthe{lemma}

\newaliascnt{corollary}{theorem}
\newtheorem{corollary}[corollary]{Corollary}
\aliascntresetthe{corollary}

\newtheorem{maintheorem}{Theorem}

\theoremstyle{definition}
\newaliascnt{definition}{theorem}

\aliascntresetthe{definition}

\newtheorem*{question}{Question}

\theoremstyle{definition}

\AtBeginEnvironment{proof}{\vspace{6pt}}

\makeatletter
\AtBeginEnvironment{thebibliography}{\def\@mklab#1{#1\hfil}}
\makeatother

\usepackage[toc,page]{appendix}

\newcommand{\bR}{\mathbb{R}}
\newcommand{\bC}{\mathbb{C}}
\newcommand{\bZ}{\mathbb{Z}}

\newcommand{\bT}{\mathbb{T}}

\newcommand{\E}{\mathbb{E}}
\newcommand{\Pp}{\mathbb{P}}
\DeclareMathOperator{\Var}{Var}
\DeclareMathOperator{\Cov}{Cov}

\DeclareMathOperator{\tr}{tr}

\newcommand{\eps}{\varepsilon}
\newcommand{\ind}{\mathbbm{1}}

\newcommand{\dto}{\xrightarrow{\,d\,}}
\newcommand{\pto}{\xrightarrow{\,\Pp\,}}
 
\begin{document}

\title{Completely Positive Entropy and Fourier Central Limit Theorems for Stationary Random Measures}

\author{Michael Bj\"orklund}
\address{Department of Mathematics, Chalmers University of Technology and University of Gothenburg, Gothenburg, Sweden}
\email{micbjo@chalmers.se}

\begin{abstract}
We prove an almost-everywhere Fourier central limit theorem for stationary
random measures on $\bR^d$ with local second moments whose translation action
is essentially free and has completely positive entropy. For the resulting
almost-everywhere defined Bartlett density $s_\eta$, we show that there is a
single $\lambda_d$-conull set of frequencies, independent of the test
functions, on which finite collections of normalized smooth-window Fourier
transforms converge jointly to proper complex Gaussian limits with covariance
determined by $s_\eta$. No quantitative mixing, correlation-decay, or
cumulant-summability assumption is imposed. For stationary point processes of
positive intensity, the same good-frequency set yields Gaussian limits for
ball-window Fourier transforms and exponential limits for their squared
moduli. We also construct a stationary ergodic zero-entropy random measure with
bounded continuous Bartlett density, positive $\lambda_1$-almost everywhere,
for which the Fourier central limit theorem fails.
\end{abstract}

\maketitle

\section{Introduction}
\label{sec:introduction}

Write $\lambda_d$ for Lebesgue measure on $\bR^d$, and let
$\mathcal M_+(\bR^d)$ denote the space of locally finite positive Radon
measures. Let $\eta$ be a translation-invariant Borel probability measure on
$\mathcal M_+(\bR^d)$ with local second moments, and let $\omega$ denote the
canonical random measure with law $\eta$. Its intensity $\rho_\eta$ is defined
by
\[
    \E_\eta[\omega(f)]=\rho_\eta\int_{\bR^d}f\,d\lambda_d,
    \qquad f\in C_c(\bR^d).
\]
Set
\[
    M_\omega:=\omega-\rho_\eta\lambda_d,
\]
and write $M(f)$ for the random variable $\omega\mapsto M_\omega(f)$. For bounded compactly supported Borel $f$, write
\[
    (\mathbb S f)(\omega):=\omega(f),
\]
and use the same notation for Schwartz functions. The \emph{Bartlett spectrum}
$\sigma_\eta$ is characterized, for $f\in\mathcal S(\bR^d)$, by
\begin{equation}
\label{eq:intro-bartlett}
    \Var_\eta(\mathbb S f)
    =\int_{\bR^d}|\hat f(\xi)|^2\,d\sigma_\eta(\xi).
\end{equation}
Under the CPE hypotheses imposed below, $\sigma_\eta\ll\lambda_d$
automatically; see Corollary~\ref{cor:cpe-bartlett-ac}. We write
\begin{equation}
\label{eq:intro-bartlett-density}
    s_\eta:=\frac{d\sigma_\eta}{d\lambda_d}.
\end{equation}
The Bartlett density is an element of $L^1_{\mathrm{loc}}(\lambda_d)$ defined
only modulo $\lambda_d$-null sets. Thus a pointwise value $s_\eta(\xi)$ is not
part of the spectral data at an exceptional frequency; statements involving
$s_\eta(\xi)$ are intrinsically $\lambda_d$-almost-everywhere statements unless
additional regularity specifies a pointwise version.

For $h\in C_c^\infty(\bR^d;\bC)$, let
\begin{equation}
\label{eq:intro-empirical-fourier}
    Z_R(\xi;h)
    :=R^{-d/2}M\bigl(h(\cdot/R)\chi_\xi\bigr).
\end{equation}
We refer to the variables in~\eqref{eq:intro-empirical-fourier} as the
normalized empirical Fourier transforms of $M$. For the CPE class considered
below, the Bartlett density therefore exists, and for $\lambda_d$-almost every
$\xi$ the Bartlett identity gives
\begin{equation}
\label{eq:intro-covariance-target}
    \E_\eta\!\left[Z_R(\xi;h)\overline{Z_R(\xi;g)}\right]
    \longrightarrow
    s_\eta(\xi)\int_{\bR^d}h\overline g\,d\lambda_d.
\end{equation}
Equation~\eqref{eq:intro-covariance-target} identifies their asymptotic second
moments. We ask for conditions under which these random variables have Gaussian
limits.

The qualifier ``for $\lambda_d$-almost every $\xi$'' is essential. The results
below make no assertion at a prescribed frequency, since that frequency may
belong to the exceptional null set. In particular, they give no conclusion at
$\xi=0$, the distinguished frequency in the study of number fluctuations and
hyperuniformity. The theorem is an almost-everywhere-frequency CLT, not a
fixed-frequency CLT.

Finite Fourier transforms of point processes go back at least to Bartlett's
spectral analysis \cite{Bar63}. Classical continuum Fourier central limit
theorems were developed by Brillinger under higher-order spectral assumptions:
\cite{Bri72} treats stationary interval functions, including stationary random
measures on $\bR$, and \cite{Bri82} proves asymptotic normality of finite
Fourier transforms of stationary generalized processes over locally compact
abelian groups under cumulant-spectrum mixing conditions. The latter theorem
works at prescribed distinct frequencies where the second-order spectrum is
continuous and positive. Thus the result below is not the first continuum
random-measure Fourier CLT; its contribution is to obtain a continuum theorem
at Lebesgue-almost every frequency from CPE and local second moments, without higher-order
cumulant assumptions. Squared Fourier moduli are standard empirical spectral
quantities and are used in recent work on the structure factor and
hyperuniformity; see, for example, \cite{HGBL23,KLH26,MBL24}.

Almost-everywhere Fourier central limit theorems under qualitative projective
or filtration assumptions form a separate line of work. Peligrad and Wu
\cite{PW10} prove a Fourier CLT for regular stationary sequences at
Lebesgue-almost every frequency by a frequency-dependent martingale
approximation. Cohen and Conze \cite{CC13} give a Fej\'er--Lebesgue proof in
the setting of $K$-systems and extend the method to $\bZ^d$-actions. Peligrad
and Zhang \cite{PZ19} prove a Lebesgue-almost-everywhere Fourier CLT for
stationary random fields adapted to a commuting filtration; their assumptions
impose no rate of covariance decay and no smoothness of the spectral density.
Section~\ref{sec:lattice-clt} proves the corresponding lattice theorem for
Hilbert-valued fields with smooth $d$-dimensional weights. The discrete
results just cited do not by themselves give the continuum random-measure
theorem: the cell reduction must retain the full centered restriction of the
random measure in one fixed Hilbert-valued field, and the resulting lattice
spectrum is a periodization of the continuum Bartlett spectrum. Accordingly,
Section~\ref{sec:continuum} requires three additional arguments: the unit-cell
data are distribution-valued, lattice sampling aliases continuum frequencies
modulo $\bZ^d$, and the exceptional set in the cell spectral identity must be
chosen independently of $\xi$. These are handled by the $H^{-s}$ unit-cell
encoding, the identity $\widehat{\ind}_C(m)=0$ for
$m\in\bZ^d\setminus\{0\}$, and the common-set spectral identification in
Proposition~\ref{prop:cell-periodization-dealiasing}, which recovers
$s_\eta(\xi)$ at the original continuum frequency on one conull set.
Section~\ref{sec:scattering} gives sharp-window and point-process limits, while
Theorem~\ref{thm:intro-counterexample} shows that a bounded continuous Bartlett
density alone does not imply the Fourier CLT. The perfect-past filtration used
below is lexicographically ordered rather than commuting in the coordinatewise
sense of \cite{PZ19}, so the lattice Fourier CLT is proved directly rather than
invoked from that result.

Other central limit theorems for spatial systems start from quantitative
dependence assumptions. Bolthausen \cite{Bol82} treats stationary random fields
under strong mixing conditions. The author and Gorodnik \cite{BG20} prove CLTs
for group actions that are exponentially mixing of all orders by a cumulant
method. In the point-process setting, Klatt, Last and Henze \cite{KLH26} obtain
a multivariate Fourier CLT from factorial-cumulant hypotheses and use the
resulting scattering laws for hyperuniformity testing. Krishnapur and
Yogeshwaran \cite{KY24} prove CLTs for smooth statistics of simple point
processes satisfying an integral identity for higher-order truncated
correlation functions. More generally, B{\l}aszczyszyn, Yogeshwaran and
Yukich \cite{BYY26} derive Gaussian limits for statistics of marked point
processes from fast mixing of correlations; their cumulant argument passes
through Brillinger mixing and can also recover the Klatt--Last--Henze Fourier
CLT; see \cite[p.~61]{BYY26}. The hypothesis in the present paper is
different: no quantitative mixing rate, correlation-decay estimate, or
cumulant-summability assumption is imposed. CPE is nevertheless a strong
qualitative dynamical hypothesis; the distinction here is that the argument
requires no quantitative rate of decorrelation.

Our assumption is that the translation action has \emph{completely positive
entropy} (CPE). For the standard-lattice subaction, the Rohlin--Sinai theorem
in dimension one and Kami\'nski's higher-rank theory of perfect
$\sigma$-algebras \cite{Kam81,Kam91,Kam96} provide a lexicographically ordered
past whose remote tail is the Pinsker algebra. Under CPE this tail is trivial.
Section~\ref{sec:perfect-past} then obtains an orthogonal decomposition
of $L^2_0$ into lexicographically ordered martingale differences, which is the
input to the Fourier CLT. Bernoulli actions are CPE, but CPE is strictly weaker
than the Bernoulli property; the precise entropy conventions are given in
Section~\ref{subsec:entropy-cpe}. This gives a measure-theoretic alternative to model-specific quantitative
decorrelation assumptions.

The CPE hypothesis includes several standard point-process models. Homogeneous
Poisson point processes are Bernoulli under translations, and Bernoulli
$\bR^d$-actions are CPE in the spatial-entropy sense used below; see the
closing remark of \cite[Section~3.1]{DG12}. Classical Mat\'ern hard-core
processes of types I and II give further examples. They are obtained by
translation-equivariant finite-range thinnings of homogeneous Poisson or
independently marked Poisson processes \cite{TBvB13}, and hence are CPE as
factors of Bernoulli actions. Their finite-range dependence implies mixing and,
for these nontrivial processes, essential freeness: ergodicity makes an
almost-sure stabilizer deterministic, while mixing excludes a nonzero
deterministic period. The same argument applies to many nontrivial
Poisson-driven hard-core and soft-core thinning models defined by finite-range
local rules.

CPE is also compatible with nonintegrable covariance. A standard example
is the stationary infinite-source Poisson process \cite{FRS07}. Let
$\{(T_i,L_i)\}$ be a homogeneous marked Poisson process on
$\bR\times(0,\infty)$ with ground intensity $\lambda$, and set
\[
    X(t):=\sum_i \ind_{[T_i,T_i+L_i]}(t).
\]
If $\Pp(L>u)\sim c u^{-\alpha}$ with $1<\alpha<2$, then
\[
    \Cov(X(0),X(t))
    =\lambda\E[(L-|t|)_+]
    \sim \frac{\lambda c}{\alpha-1}|t|^{1-\alpha},
\]
so the covariance is not integrable. Since $\alpha>1$, one has
$\E L<\infty$. Moreover, if $B$ is a bounded interval of length $T$ and
$N_B$ denotes the number of source intervals intersecting $B$, then
$N_B$ is Poisson with mean $\lambda(T+\E L)$ and
\[
    \int_B X(t)\,dt \leq T N_B.
\]
Thus the associated random measure $X(t)\,dt$ has local moments of every
order. It is an equivariant factor of the marked Poisson process and is
therefore CPE. As a nontrivial factor of a Bernoulli action it is also mixing,
and the stabilizer argument used above for the Mat\'ern models gives essential
freeness. Hence Theorem~\ref{thm:intro-main-random-measure} applies to a CPE
model with nonsummable pair correlations; in particular, CPE does not entail
short-range dependence. One may also use an independent Poisson process on
$\bR\times\bR_+$ and retain $(t,u)$ when $u\leq a+bX(t)$ to obtain a stationary
Cox point process. This is again a nontrivial factor of a Bernoulli action,
hence CPE, mixing, and essentially free; for $t\neq0$, its pair-covariance
density is $b^2\Cov(X(0),X(t))$, and therefore has the same nonintegrable
power-law tail.

Osada \cite[Theorem~1.1]{Osa21} proved the Bernoulli property for
translation-invariant determinantal point processes for which the Fourier
transform of the kernel is integrable and takes values in $[0,1]$ almost
everywhere. For stationary subcritical linear Hawkes processes on $\bR$, the
Poisson-cluster representation of Hawkes and Oakes \cite{HO74} realizes the
process as an equivariant factor of a marked Poisson process (the mark records
the descendant cluster); hence its translation action is CPE. These nontrivial
Bernoulli and Poisson-factor examples are mixing, and therefore essentially
free by the same stabilizer argument as above.

CPE is genuinely weaker than the Bernoulli property even among stationary
point processes. Smorodinsky adapted Ornstein's construction to obtain a
$K$-flow which is not Bernoulli \cite{Smo75}. By the generating Delone
cross-section construction in
\cite{Bjo26b}, such a flow admits a return-time point process measurably
isomorphic to the original flow. The resulting stationary point process is CPE
but not Bernoulli.

\subsection{Main results}
\label{subsec:introduction-main-results}

\begin{maintheorem}[Empirical Fourier transforms under CPE]
\label{thm:intro-main-random-measure}
Let $\eta$ be a stationary random measure on $\bR^d$ with local second
moments. Assume that its translation action is essentially free and has CPE.
Then $\sigma_\eta\ll\lambda_d$. Moreover, for $\lambda_d$-almost every
$\xi\in\bR^d$, simultaneously for every $m\geq1$ and every
$h_1,\ldots,h_m\in C_c^\infty(\bR^d;\bC)$,
\begin{equation}
\label{eq:intro-main-clt}
    \bigl(Z_R(\xi;h_1),\ldots,Z_R(\xi;h_m)\bigr)
    \dto (Z_1,\ldots,Z_m),
\end{equation}
where $(Z_1,\ldots,Z_m)$ is a complex Gaussian vector with mean zero and
\begin{equation}
\label{eq:intro-main-covariance}
    \E[Z_a\overline{Z_b}]
    =s_\eta(\xi)\int_{\bR^d}h_a\overline{h_b}\,d\lambda_d,
    \qquad \E[Z_aZ_b]=0.
\end{equation}
\end{maintheorem}

For a stationary point process of positive intensity $\rho_\eta$, define the
structure-factor density
\[
    S_\eta:=\frac{s_\eta}{\rho_\eta},
\]
again only modulo $\lambda_d$-null sets. For point processes, essential
freeness means that for $\eta$-almost every configuration $\omega$, no
nonzero translation leaves $\omega$ invariant.

\begin{corollary}[Sharp-window scattering limit]
\label{cor:intro-point-process-limits}
Under the assumptions of Theorem~\ref{thm:intro-main-random-measure}, if
$\eta$ is a stationary point process of positive intensity, then for
$\lambda_d$-almost every $\xi\in\bR^d$,
\begin{equation}
\label{eq:intro-point-process-fourier-limit}
    \frac{1}{\sqrt{|\omega\cap B_R|}}
      \sum_{x\in\omega\cap B_R}\chi_\xi(x)
    \dto \mathcal{CN}\bigl(0,S_\eta(\xi)\bigr),
\end{equation}
where the left-hand side is defined to be $0$ when $|\omega\cap B_R|=0$,
and consequently
\begin{equation}
\label{eq:intro-point-process-scattering-limit}
    \frac{1}{|\omega\cap B_R|}
      \left|\sum_{x\in\omega\cap B_R}\chi_\xi(x)\right|^2
    \dto S_\eta(\xi)\operatorname{Exp}(1).
\end{equation}
The second left-hand side is also defined to be $0$ when
$|\omega\cap B_R|=0$.
\end{corollary}

Theorem~\ref{thm:intro-main-random-measure} gives CPE as a sufficient
condition; we do not address necessity. Absolute continuity of the Bartlett
spectrum alone would be a weaker hypothesis. The next theorem shows that it is
insufficient even when the density is bounded, continuous, and positive
$\lambda_1$-almost everywhere.

\begin{maintheorem}[Regular Bartlett spectrum without a Fourier CLT]
\label{thm:intro-counterexample}
There is a stationary ergodic random measure $\eta$ on $\bR$ with zero entropy
and a bounded continuous Bartlett density $s_\eta$, positive
$\lambda_1$-almost everywhere, together with a sequence $R_N\to\infty$, such
that for every $\xi$ with $s_\eta(\xi)>0$,
\begin{equation}
\label{eq:intro-counterexample-fourier}
    R_N^{-1/2}M\bigl(1_{[0,R_N]}\chi_\xi\bigr)
\end{equation}
does not converge to any centered complex Gaussian law; in particular, it does
not converge to $\mathcal{CN}(0,s_\eta(\xi))$. Its squared modulus does not
converge to $s_\eta(\xi)\operatorname{Exp}(1)$. Moreover, for every such
$\xi$ there is $h\in C_c^\infty(\bR;\bC)$ for which
$R^{-1/2}M(h(\cdot/R)\chi_\xi)$ does not converge, as $R\to\infty$, to
$\mathcal{CN}(0,s_\eta(\xi)\lVert h\rVert_2^2)$.
\end{maintheorem}

Thus the Bartlett spectrum determines the asymptotic second moments of the
empirical Fourier transforms, but not their limiting law. Non-Gaussian limits
also occur for point processes in a different scaling regime. Mastrilli
\cite{Mas26} proves Poisson-integral and stable limits for smooth large-scale
linear statistics of one-dimensional independently perturbed lattices. These
statistics probe the low-frequency regime rather than a fixed nonzero Fourier
frequency, and are therefore complementary to Theorem~\ref{thm:intro-counterexample}.

Section~\ref{sec:counterexample} starts from the classical Rudin--Shapiro
sequence. Its $\{\pm1\}$-valued coordinate process has autocorrelation
$\delta_{m,0}$, while Balister's estimate bounds every finite exponential sum
by $\sqrt{10N}$. A constant-roof suspension produces a stationary random
measure with an explicit sinc-squared Bartlett density. Its normalized Fourier
transforms have the second moments prescribed by this density but remain in a
fixed bounded disk, so they cannot converge to the corresponding nondegenerate
complex Gaussian. Whether an analogous fixed-frequency counterexample can be
realized by a stationary point process is left open.

\subsection{Cross-sections and a realization question}
\label{subsec:introduction-cross-sections}

Let $\bR\curvearrowright(X,\mu)$ be an essentially free ergodic p.m.p. flow and
let $Y\subset X$ be a separated cross-section. Its return-time process is
\begin{equation}
\label{eq:intro-cross-section-coordinates}
    Y_x:=\{t\in\bR:t.x\in Y\},\qquad
    \kappa_Y(x):=\delta_{Y_x},\qquad
    \eta_Y:=(\kappa_Y)_*\mu.
\end{equation}
If the original flow has CPE and the return-time process has local second
moments, then the point-process factor is CPE. It is also essentially free. For
a nonempty locally finite configuration, the stabilizer is a proper closed
subgroup of $\bR$, hence either $\{0\}$ or $p\bZ$ for some $p>0$. Thus if a
nontrivial stabilizer occurred, ergodicity would give a deterministic period
$p>0$ almost surely, so the factor action would descend to an ergodic action of
the compact group $\bR/p\bZ$ and therefore have discrete spectrum and zero
entropy. The return-time factor is nontrivial because a cross-section meets
almost every orbit, so this would contradict CPE. Thus
Theorem~\ref{thm:intro-main-random-measure} applies: every separated
cross-section of a CPE flow whose return-time process has local second moments
satisfies the Fourier CLT for $\lambda_1$-almost every frequency.

Burton and Denker \cite{BD87} proved that every aperiodic
probability-preserving transformation admits a centered square-integrable
observable whose ergodic sums satisfy a nondegenerate central limit theorem.
Voln\'y \cite{Vol99} proved an invariance-principle version of this result, and
Kosloff and Voln\'y \cite{KV22} showed that, for every ergodic aperiodic
transformation, one can choose a square-integrable integer-valued observable
satisfying a lattice local central limit theorem. Thus these limit laws can be
realized by a suitable choice of observable without a mixing assumption on the
ambient transformation.

In another direction, in \cite{Bjo26a,Bjo26b} the author constructs generating
Delone cross-sections whose return-time processes have strong low-frequency
spectral or variance properties, including a spectral gap around the origin
(stealthiness) and hyperuniformity.
A return-time point process is much more constrained than an arbitrary
observable. This leads to the following question.

\begin{question}
Does every essentially free ergodic p.m.p. $\bR$-flow admit a separated
cross-section, perhaps generating, whose return-time point process has local
second moments and satisfies a Fourier central limit theorem for
$\lambda_1$-almost every frequency?
\end{question}

\subsection{From CPE to the continuum Fourier CLT}
\label{subsec:introduction-outline}

By the lattice-subaction theorem of Dooley and Golodets
\cite[Theorem~1.3]{DG12}, the standard-lattice subaction of an essentially
free CPE $\bR^d$-action is CPE.
Section~\ref{sec:perfect-past} applies the Rohlin--Sinai--Kami\'nski
perfect-past theorem to obtain
\begin{equation}
\label{eq:intro-wandering-decomposition}
    L^2_0(X,\mu)=\bigoplus_{k\in\bZ^d}U^kK_0,
    \qquad
    K_0=L^2(\mathcal A_0)\ominus L^2(\mathcal A_{-e_d}),
\end{equation}
with $U^kK_0$ a martingale-difference family in lexicographic order.
Section~\ref{sec:lattice-clt} proves the lattice Fourier CLT from this
decomposition. An $L^2$ martingale approximation and McLeish's theorem
give the Gaussian limit for Lebesgue-almost every frequency in $\bT^d$.
Section~\ref{sec:continuum} partitions $\bR^d$ into unit cells and applies the
lattice theorem to the centered restrictions of the random measure.
Proposition~\ref{prop:cell-periodization-dealiasing} identifies the variance
at a continuum frequency, and Proposition~\ref{prop:cell-freezing} shows that,
after normalization, replacing $h((u-k)/R)$ on each unit cell by $h(-k/R)$
changes $Z_R(\xi;h)$ by a term converging to zero in probability. These facts
prove Theorem~\ref{thm:intro-main-random-measure}.
Section~\ref{sec:scattering} approximates ball indicators by smooth functions
and proves Corollary~\ref{cor:intro-point-process-limits}.
Section~\ref{sec:counterexample} proves
Theorem~\ref{thm:intro-counterexample}.

\subsection{Acknowledgments}
\label{subsec:acknowledgments}

The author thanks G\"unter Last for explaining his joint work with Michael A.~Klatt
and Norbert Henze at the HSRPP workshop in Lille in 2023, and for subsequent
discussions concerning the approach developed here. Earlier, during a visit to
Jerusalem in 2010, Jean-Pierre Conze explained his joint work with Guy Cohen on
rotated central limit theorems; those discussions provided an early motivation for the
perspective taken in this paper.

During preparation of the manuscript, the author used ChatGPT (OpenAI) for
editorial assistance and consistency checking, and Claude (Anthropic), Gemini
(Google DeepMind), and Aristotle (Harmonic) for additional manuscript and proof
checks. Aristotle was also used to check isolated algebraic and analytic
computations. All AI-assisted output was reviewed and verified by the author,
who takes full responsibility for the mathematical content, references, and
final text.

 \section{Preliminaries}
\label{sec:preliminaries}

\subsection{Fourier conventions}
\label{subsec:fourier-conventions}

We write $\lambda_d$ for Lebesgue measure on $\bR^d$, $B_R$ for the open ball
of radius $R$ centered at the origin, and $\mathcal S(\bR^d)$ for the Schwartz
space. For $f\in L^1(\bR^d)$, we use the Fourier transform
\begin{equation}
\label{eq:fourier-convention}
    \hat f(\xi)
    :=\int_{\bR^d}f(x)e^{-2\pi i\langle x,\xi\rangle}\,d\lambda_d(x),
    \qquad \xi\in\bR^d.
\end{equation}
Thus Plancherel's theorem holds without an additional constant. We write
$\chi_\xi(x):=e^{2\pi i\langle x,\xi\rangle}$ for the character associated
with $\xi$. If $h_R(x):=h(x/R)$ and $f_{R,\xi}:=h_R\chi_\xi$, then scaling and
modulation give
\begin{equation}
\label{eq:scaled-modulated-fourier}
    \hat f_{R,\xi}(\lambda)
    =R^d\hat h\bigl(R(\lambda-\xi)\bigr).
\end{equation}
This is the normalization used throughout the paper.

For the lattice arguments in Sections~\ref{sec:perfect-past} and
\ref{sec:lattice-clt}, we identify
$\bT^d=\bR^d/\bZ^d$ with $[-1/2,1/2)^d$ and write $d\theta$ for normalized
Lebesgue measure on this fundamental domain. Its characters are
$\theta\mapsto e^{2\pi i\langle k,\theta\rangle}$, $k\in\bZ^d$. Let $p:\bR^d\to\bT^d$ be the quotient map, write $[\xi]:=p(\xi)$, and set $\|\theta\|_{\bT^d}:=\inf_{m\in\bZ^d}\|\theta-m\|$.

\subsection{Stationary random measures and the Bartlett spectrum}
\label{subsec:random-measures-bartlett}

Let $\mathcal M_+(\bR^d)$ be the space of locally finite positive Radon
measures, equipped with the vague topology. Translations act by
\begin{equation}
\label{eq:translation-action-random-measures}
    (v.\omega)(f):=\omega\bigl(f(\,\cdot-v)\bigr),
    \qquad v\in\bR^d,\quad f\in C_c(\bR^d),
\end{equation}
so that $v.\delta_P=\delta_{P-v}$ for every locally finite set $P$. A
\emph{stationary random measure} is a translation-invariant Borel probability
measure $\eta$ on $\mathcal M_+(\bR^d)$.

We say that $\eta$ has \emph{local second moments} if
\[
    \int_{\mathcal M_+(\bR^d)}\omega(A)^2\,d\eta(\omega)<\infty
\]
for every bounded Borel set $A\subset\bR^d$. Its first moment measure is then
translation invariant, and hence there is a unique constant $\rho_\eta\geq0$
such that
\begin{equation}
\label{eq:intensity-definition}
    \int_{\mathcal M_+(\bR^d)}\omega(f)\,d\eta(\omega)
    =\rho_\eta\int_{\bR^d}f\,d\lambda_d,
    \qquad f\in C_c(\bR^d).
\end{equation}
The number $\rho_\eta$ is the \emph{intensity} of $\eta$.

For a bounded compactly supported Borel function $f$, write
\[
    (\mathbb S f)(\omega):=\omega(f).
\]
We use the same notation for Schwartz functions. Local second moments and
stationarity imply $\mathbb S f\in L^2(\eta)$ for every
$f\in\mathcal S(\bR^d)$. Indeed, a unit-cube decomposition and Minkowski's inequality give, for every
fixed $N>d$,
\begin{equation}
\label{eq:schwartz-linear-statistic-L2-bound}
    \|\mathbb S f\|_{L^2(\eta)}
    \leq C_N\sup_{x\in\bR^d}(1+\|x\|)^N|f(x)|.
\end{equation}
In particular, $f\mapsto\mathbb S f$ is continuous from $\mathcal S(\bR^d)$ to
$L^2(\eta)$.
For an integrable random variable $F$, put
$F^\circ:=F-\int F\,d\eta$. Define the centered random measure
\begin{equation}
\label{eq:centered-random-measure}
    M_\omega:=\omega-\rho_\eta\lambda_d,
    \qquad M(f)=(\mathbb S f)^\circ.
\end{equation}

A Radon measure $\nu$ on $\bR^d$ is called \emph{translation bounded} if
$\sup_{a\in\bR^d}|\nu|(a+K)<\infty$ for every compact set
$K\subset\bR^d$.

\begin{proposition}[Bartlett spectrum]
\label{prop:bartlett-spectrum}
Let $\eta$ be stationary with local second moments. There is a unique positive
translation-bounded Radon measure $\sigma_\eta$ on $\bR^d$ such that
\begin{equation}
\label{eq:bartlett-covariance}
    \int (\mathbb S f)^\circ\,\overline{(\mathbb S g)^\circ}\,d\eta
    =\int_{\bR^d}\hat f(\xi)\overline{\hat g(\xi)}\,
      d\sigma_\eta(\xi),
    \qquad f,g\in\mathcal S(\bR^d).
\end{equation}
In particular,
\begin{equation}
\label{eq:bartlett-variance}
    \Var_\eta(\mathbb S f)
    =\int_{\bR^d}|\hat f(\xi)|^2\,d\sigma_\eta(\xi).
\end{equation}
The measure $\sigma_\eta$ is symmetric under $\xi\mapsto-\xi$.
\end{proposition}

\begin{proof}
By \eqref{eq:schwartz-linear-statistic-L2-bound}, the covariance form of
$f\mapsto(\mathbb S f)^\circ$ is continuous on $\mathcal S(\bR^d)$. It is positive
semidefinite and invariant under simultaneous translations of the two
functions. By the Schwartz kernel theorem and translation invariance, there is
a positive-definite tempered distribution $C$ on $\bR^d$ such that, with
$\widetilde g(x):=\overline{g(-x)}$, the covariance form is
$C(f*\widetilde g)$. The Bochner--Schwartz theorem then gives a unique positive
tempered Radon measure satisfying \eqref{eq:bartlett-covariance}; compare
\cite[Proposition~8.2.I and Definition~8.2.II]{DVJ03} in the point-process
setting. For real $f,g$ the covariance is real and symmetric, so $C$ is even;
its Fourier measure $\sigma_\eta$ is therefore symmetric under
$\xi\mapsto-\xi$.

It remains to record translation boundedness. Choose
$\varphi\in C_c^\infty(\bR^d)$ and a neighborhood $Q$ of the origin such that
$|\hat\varphi|\geq c>0$ on $Q$. For $a\in\bR^d$, put
$\varphi_a:=\chi_a\varphi$. Since
$\hat\varphi_a(\xi)=\hat\varphi(\xi-a)$, the Bartlett identity gives
\begin{equation}
\label{eq:bartlett-translation-bounded-proof}
    c^2\sigma_\eta(a+Q)
    \leq \Var_\eta(\mathbb S\varphi_a)
    \leq \int (\mathbb S|\varphi|)^2\,d\eta.
\end{equation}
The last quantity is independent of $a$, and hence
$\sup_a\sigma_\eta(a+Q)<\infty$. Every compact set is covered by finitely many
translates of $Q$, which gives translation boundedness.
\end{proof}

We call $\sigma_\eta$ the \emph{Bartlett spectrum} of $\eta$. When it is
absolutely continuous, we write
\begin{equation}
\label{eq:bartlett-density}
    d\sigma_\eta(\xi)=s_\eta(\xi)\,d\lambda_d(\xi).
\end{equation}
The \emph{Bartlett density} $s_\eta$ is an $L^1_{\mathrm{loc}}(\lambda_d)$ function
defined only up to $\lambda_d$-null sets. Translation boundedness of
$\sigma_\eta$ becomes a uniform local $L^1$ bound for $s_\eta$.

Whenever a representative $s\in L^1_{\mathrm{loc}}(\lambda_d)$ is fixed,
we call $\xi\in\bR^d$ a \emph{Lebesgue point} of $s$ if
\begin{equation}
\label{eq:lebesgue-point-definition}
    \lim_{r\downarrow0}
    \frac{1}{\lambda_d(B_r)}
    \int_{B_r(\xi)}|s(u)-s(\xi)|\,d\lambda_d(u)=0.
\end{equation}
By the Lebesgue differentiation theorem, this holds for $\lambda_d$-almost
every $\xi$. This convention is used whenever a Bartlett density is evaluated
pointwise.

The Bartlett identity extends to indicators of bounded regular sets. This is
the form used later for sharp observation windows.

\begin{lemma}
\label{lem:bartlett-indicators}
Let $A,B\subset\bR^d$ be bounded Borel sets with
$\lambda_d(\partial A)=\lambda_d(\partial B)=0$. Then
\begin{equation}
\label{eq:bartlett-indicator-covariance}
    \Cov_\eta\bigl(\mathbb S\ind_A,\mathbb S\ind_B\bigr)
    =\int_{\bR^d}\widehat{\ind}_A(\xi)
       \overline{\widehat{\ind}_B(\xi)}\,d\sigma_\eta(\xi).
\end{equation}
In particular, $\Var_\eta(\mathbb S\ind_A)$ is obtained by taking $A=B$.
For every $g\in\mathcal S(\bR^d)$, one also has the mixed identity
\begin{equation}
\label{eq:bartlett-indicator-schwartz-covariance}
    \Cov_\eta\bigl(\mathbb S\ind_A,\mathbb S g\bigr)
    =\int_{\bR^d}\widehat{\ind}_A(\xi)
      \overline{\hat g(\xi)}\,d\sigma_\eta(\xi).
\end{equation}
Consequently, the Bartlett covariance identity holds on the linear span of
Schwartz functions and indicators of bounded Borel sets with boundary of
Lebesgue measure zero. Moreover, for every $\xi\in\bR^d$, the centered statistic
$M(\ind_A\chi_\xi)$ has Koopman spectral measure
\begin{equation}
\label{eq:bartlett-indicator-spectral-measure}
    d\varsigma_{M(\ind_A\chi_\xi)}(\lambda)
    =\bigl|\widehat{\ind}_A(\lambda-\xi)\bigr|^2
      \,d\sigma_\eta(\lambda).
\end{equation}
\end{lemma}

\begin{proof}
Choose smooth compactly supported functions $0\leq\varphi_n,\psi_n\leq1$
converging to the two indicators pointwise off their boundaries and in
$L^1(\lambda_d)$, with all supports contained in fixed compact sets. By
\eqref{eq:intensity-definition},
$\omega(\partial A)=\omega(\partial B)=0$ for $\eta$-almost every $\omega$.
Local second moments then give convergence of the corresponding linear
statistics in $L^2(\eta)$. Their Fourier transforms converge uniformly to the Fourier transforms of the
indicators, while \eqref{eq:bartlett-covariance} makes them Cauchy in
$L^2(\sigma_\eta)$. A subsequence therefore converges $\sigma_\eta$-almost
everywhere to its $L^2(\sigma_\eta)$ limit, which identifies that limit with
the corresponding indicator Fourier transform. Passing to the limit gives
\eqref{eq:bartlett-indicator-covariance}. Keeping the second argument equal to
a fixed $g\in\mathcal S(\bR^d)$ and passing to the limit only in the first
argument gives \eqref{eq:bartlett-indicator-schwartz-covariance}; on the
spectral side this follows by dominated convergence, since $\sigma_\eta$ is
translation bounded and $\hat g$ is rapidly decreasing. Bilinearity then gives
the stated extension to the linear span. Applying the same
approximation to $\varphi_n\chi_\xi$ and using the Bartlett--Koopman identity
for the smooth approximants gives
\eqref{eq:bartlett-indicator-spectral-measure}.
\end{proof}

Let $\mathcal N_s(\bR^d)\subset\mathcal M_+(\bR^d)$ be the space of locally
finite simple counting measures. A stationary random measure supported on
$\mathcal N_s(\bR^d)$ is a \emph{stationary point process}. If $\rho_\eta>0$ and
\eqref{eq:bartlett-density} holds, we call
\begin{equation}
\label{eq:structure-factor-definition}
    S_\eta(\xi):=\frac{s_\eta(\xi)}{\rho_\eta}
\end{equation}
the \emph{structure factor}. Like $s_\eta$, it is intrinsically an
a.e.-defined function.

\subsection{Empirical Fourier transforms}
\label{subsec:empirical-fourier-preliminaries}

Let $h\in C_c^\infty(\bR^d;\bC)$ and put $h_R(x):=h(x/R)$. For
$\xi\in\bR^d$, the \emph{empirical Fourier transform} of the centered random
measure at scale $R$ is
\begin{equation}
\label{eq:empirical-fourier-definition}
    Z_R(\xi;h)
    :=R^{-d/2}M(h_R\chi_\xi).
\end{equation}
The normalization is chosen so that, at Lebesgue points of a Bartlett density,
the variance has a finite limit. Indeed, Proposition~\ref{prop:bartlett-spectrum}
and \eqref{eq:scaled-modulated-fourier} give, for
$h,g\in C_c^\infty(\bR^d;\bC)$,
\begin{equation}
\label{eq:empirical-fourier-second-moment-measure}
    \E_\eta\!\left[Z_R(\xi;h)\overline{Z_R(\xi;g)}\right]
    =R^d\int_{\bR^d}
      \hat h\bigl(R(\lambda-\xi)\bigr)
      \overline{\hat g\bigl(R(\lambda-\xi)\bigr)}\,
      d\sigma_\eta(\lambda).
\end{equation}
If $d\sigma_\eta=s_\eta\,d\lambda_d$, the change of variables
$u=R(\lambda-\xi)$ turns this into
\begin{equation}
\label{eq:empirical-fourier-second-moment-density}
    \E_\eta\!\left[Z_R(\xi;h)\overline{Z_R(\xi;g)}\right]
    =\int_{\bR^d}\hat h(u)\overline{\hat g(u)}\,
      s_\eta(\xi+u/R)\,d\lambda_d(u).
\end{equation}

The following differentiation lemma will also be used for ball windows in
Section~\ref{sec:scattering}.

\begin{lemma}[Differentiation against decaying kernels]
\label{lem:differentiation-decaying-kernels}
Let $s\geq0$ be locally integrable and suppose that
$s\lambda_d$ is translation bounded. Let $\xi$ be a Lebesgue point of $s$,
and let $K:\bR^d\to\bC$ satisfy
\[
    |K(u)|\leq C(1+\|u\|)^{-d-\delta}
\]
for some $C,\delta>0$. Then
\begin{equation}
\label{eq:differentiation-decaying-kernels}
    \int_{\bR^d}K(u)s(\xi+u/R)\,d\lambda_d(u)
    \longrightarrow
    s(\xi)\int_{\bR^d}K(u)\,d\lambda_d(u).
\end{equation}
\end{lemma}

\begin{proof}
On every fixed ball, the assertion follows directly from the Lebesgue-point
property after the change of variables $u=R(\lambda-\xi)$. It remains to make
the tails uniform in $R$. At a Lebesgue point there are $r_0,C_0>0$ such that
\[
    \int_{B_r(\xi)}s\,d\lambda_d\leq C_0r^d
    \qquad (0<r\leq r_0).
\]
Translation boundedness gives the same growth, with another constant, for
$r\geq r_0$. Decomposing $\{\|u\|>A\}$ into dyadic annuli and using these two
bounds yields
\[
    \sup_{R\geq1}\int_{\|u\|>A}|K(u)|s(\xi+u/R)\,d\lambda_d(u)
    \leq C_1A^{-\delta}.
\]
The same estimate applies to the constant function $s(\xi)$. First let
$R\to\infty$ on a fixed ball and then let $A\to\infty$.
\end{proof}

Applying the lemma to
$K=\hat h\,\overline{\hat g}$ and using Plancherel gives the second-order limit
\begin{equation}
\label{eq:empirical-fourier-covariance-limit}
    \E_\eta\!\left[Z_R(\xi;h)\overline{Z_R(\xi;g)}\right]
    \longrightarrow
    s_\eta(\xi)\int_{\bR^d}h(x)\overline{g(x)}\,d\lambda_d(x)
\end{equation}
for every Lebesgue point $\xi$ of $s_\eta$. 
For a point process, the sharp-window statistic appearing in the scattering
formulation is the uncentered normalized Fourier sum
\begin{equation}
\label{eq:point-process-empirical-transform}
    J_R(\xi)
    :=\frac{\mathbb S(\ind_{B_R}\chi_\xi)}
            {\sqrt{\mathbb S\ind_{B_R}}},
\end{equation}
with the value $0$ on $\{\mathbb S\ind_{B_R}=0\}$. We call
$I_R(\xi):=|J_R(\xi)|^2$ the \emph{empirical scattering intensity}.
Section~\ref{sec:scattering} derives the limit of $J_R(\xi)$ from the
continuum Fourier CLT and the ergodic theorem.

\subsection{Translation actions and spectral measures}
\label{subsec:translation-spectra}

Let $\bR^d\curvearrowright(X,\mathcal F,\mu)$ be a probability-preserving
Borel action on a standard probability space, written $(t,x)\mapsto t.x$.
Its Koopman representation on $L^2(X,\mu)$ is strongly continuous (a
measurable unitary representation of $\bR^d$ on the separable Hilbert space
$L^2(X,\mu)$ is continuous) and is given by
\begin{equation}
\label{eq:koopman-convention}
    U_tF(x):=F((-t).x),
    \qquad t\in\bR^d.
\end{equation}
For $F\in L^2(X,\mu)$, the spectral theorem gives a unique finite positive
measure $\varsigma_F$ on $\bR^d$ such that
\begin{equation}
\label{eq:koopman-spectral-measure}
    \langle U_tF,F\rangle_{L^2(\mu)}
    =\int_{\bR^d}e^{2\pi i\langle t,\xi\rangle}\,d\varsigma_F(\xi),
    \qquad t\in\bR^d.
\end{equation}
We call $\varsigma_F$ the \emph{Koopman spectral measure} of $F$. The least
measure class dominating the spectral measures of all vectors in
$L^2_0(X,\mu)$ is the maximal spectral type of the centered Koopman
representation.

For an abstract probability-preserving $\bZ^d$-action
$T=(T^k)_{k\in\bZ^d}$, we use the convention
$U^kF:=F\circ T^{-k}$. Its spectral measures live on $\bT^d$ and are
characterized by
\begin{equation}
\label{eq:lattice-spectral-measure}
    \langle U^kF,F\rangle
    =\int_{\bT^d}e^{2\pi i\langle k,\theta\rangle}
      \,d\varsigma_F^{\mathrm{lat}}(\theta),
    \qquad k\in\bZ^d.
\end{equation}
If $T$ is the standard-lattice subaction of the $\bR^d$-action above, so that
\[
    T^kx:=k.x,\qquad k\in\bZ^d,
\]
then
\begin{equation}
\label{eq:periodization-scalar-spectral-measure}
    \varsigma_F^{\mathrm{lat}}=p_*\varsigma_F.
\end{equation}

\begin{proposition}[Bartlett--Koopman correspondence]
\label{prop:bartlett-koopman}
Let $\eta$ be stationary with local second moments. For every
$f\in\mathcal S(\bR^d)$,
\begin{equation}
\label{eq:bartlett-koopman}
    d\varsigma_{(\mathbb S f)^\circ}(\xi)
    =|\hat f(\xi)|^2\,d\sigma_\eta(\xi).
\end{equation}
For the standard-lattice subaction $T^k\omega:=k.\omega$ one therefore has
\begin{equation}
\label{eq:bartlett-koopman-lattice}
    \varsigma_{(\mathbb S f)^\circ}^{\mathrm{lat}}
    =p_*\bigl(|\hat f|^2\sigma_\eta\bigr).
\end{equation}
\end{proposition}

\begin{proof}
By \eqref{eq:translation-action-random-measures} and
\eqref{eq:koopman-convention}, $U_t(\mathbb S f)^\circ$ is the centered linear
statistic associated with $x\mapsto f(x+t)$. Its Fourier transform is
$e^{2\pi i\langle t,\xi\rangle}\hat f(\xi)$. Substituting this into
\eqref{eq:bartlett-covariance} gives \eqref{eq:bartlett-koopman};
\eqref{eq:bartlett-koopman-lattice} then follows from
\eqref{eq:periodization-scalar-spectral-measure}.
\end{proof}

Let $\mathcal H_\eta\subset L^2_0(\eta)$ be the closed invariant subspace
generated by the centered linear statistics. Choosing a Schwartz function
whose Fourier transform never vanishes and using
\eqref{eq:bartlett-koopman} shows that $\sigma_\eta$ represents the maximal
spectral type of the Koopman representation on $\mathcal H_\eta$. The
subspace $\mathcal H_\eta$ need not equal $L^2_0(\eta)$, so
$\sigma_\eta$ need not represent the maximal spectral type on all of
$L^2_0(\eta)$.

\begin{corollary}
\label{cor:lebesgue-lattice-spectrum-implies-bartlett-ac}
Suppose that, for some $f\in\mathcal S(\bR^d)$ with $\hat f$ nowhere zero,
the lattice spectral measure
$\varsigma_{(\mathbb S f)^\circ}^{\mathrm{lat}}$ is absolutely continuous with respect
to $d\theta$ on $\bT^d$. Then $\sigma_\eta$ is absolutely continuous with
respect to $\lambda_d$.
\end{corollary}

\begin{proof}
By \eqref{eq:bartlett-koopman-lattice}, the positive measure
$p_*(|\hat f|^2\sigma_\eta)$ is absolutely continuous on $\bT^d$. Restrict
$|\hat f|^2\sigma_\eta$ to each translate of the fundamental cube
$[-1/2,1/2)^d$. On such a cube the quotient map $p$ is a translation modulo
its boundary, so the restriction is absolutely continuous with respect to
Lebesgue measure. Summing over the countably many translates gives
$|\hat f|^2\sigma_\eta\ll\lambda_d$. Since $\hat f$ never vanishes,
$\sigma_\eta\ll\lambda_d$.
\end{proof}

The terms \emph{ergodic} and \emph{essentially free} have their usual
measure-theoretic meanings. For a stationary random measure they refer to the
translation action \eqref{eq:translation-action-random-measures}.

\subsection{Entropy and completely positive entropy}
\label{subsec:entropy-cpe}

Let
$T=(T^k)_{k\in\bZ^d}$ be a probability-preserving $\bZ^d$-action on a
standard probability space $(X,\mathcal F,\mu)$. For a finite measurable
partition $\alpha=\{A_1,\ldots,A_m\}$, its Shannon entropy is
\begin{equation}
\label{eq:partition-entropy}
    H_\mu(\alpha):=-\sum_{j=1}^m\mu(A_j)\log\mu(A_j),
\end{equation}
with $0\log0:=0$. If $Q_n:=\{0,\ldots,n-1\}^d$, the entropy of $\alpha$ under
$T$ is
\begin{equation}
\label{eq:Zd-partition-entropy}
    h_\mu(T,\alpha)
    :=\lim_{n\to\infty}\frac{1}{|Q_n|}
      H_\mu\!\left(\bigvee_{k\in Q_n}T^{-k}\alpha\right).
\end{equation}
The limit exists, and the entropy of the action is
$h_\mu(T):=\sup_\alpha h_\mu(T,\alpha)$, where the supremum is over finite
measurable partitions.

A factor of $T$ may be represented by a $T$-invariant sub-$\sigma$-algebra
$\mathcal G\subset\mathcal F$. Its entropy is the supremum of
$h_\mu(T,\alpha)$ over finite $\mathcal G$-measurable partitions. There is a
largest invariant sub-$\sigma$-algebra of zero entropy, the \emph{Pinsker
$\sigma$-algebra} $\Pi(T)$. The action has \emph{completely positive entropy}
(CPE) if $\Pi(T)$ is trivial modulo null sets. Equivalently, every nontrivial
factor has positive entropy.

Every factor of a CPE action is CPE, and every CPE action is ergodic.
Bernoulli actions are CPE, but the converse fails already for $\bZ$-actions:
the classical $T,T^{-1}$ transformation is a K-automorphism \cite{Mei74} but
is not loosely Bernoulli, and hence in particular is not Bernoulli
\cite{Kal82}.

For probability-preserving $\bR^d$-actions we use the Ornstein--Weiss
spatial entropy, with Lebesgue measure as Haar measure, in the convention of
Dooley and Golodets \cite{DG12}. If $\phi$ is such an action and $\alpha$ is
a finite measurable partition, write $\operatorname{sh}(\phi,\alpha)$ for
its spatial entropy. We say that $\phi$ has \emph{completely positive
entropy} if
\[
    \operatorname{sh}(\phi,\alpha)>0
\]
for every nontrivial finite measurable partition $\alpha$, where
nontriviality is understood modulo null sets. This is the definition used in
\cite[Definition~3.1]{DG12}. The finite-partition formulation immediately
passes to factors and implies ergodicity.

For the standard Borel actions considered here, essential freeness allows us
to discard an invariant null set and work on a free invariant Borel set,
without changing the measure-preserving system or its entropy. Since $\bR^d$ belongs to the class $\mathcal{ULG}$ of
\cite[Definition~1.2]{DG12}, \cite[Theorem~1.3; see also Theorem~3.2]{DG12}
specializes as follows.

\begin{proposition}[Dooley--Golodets]
\label{prop:dooley-golodets-lattice-cpe}
Let $(t,x)\mapsto t.x$ be a free ergodic probability-preserving
$\bR^d$-action, and let $\Gamma<\bR^d$ be a full-rank lattice. Then the
$\bR^d$-action has completely positive entropy if and only if its
$\Gamma$-subaction has completely positive entropy.
\end{proposition}

Only the implication from CPE of the $\bR^d$-action to CPE of its lattice
subaction is used below. For $d=1$, apart from the trivial action, CPE implies
essential freeness: a nonzero period makes an ergodic flow periodic and hence
zero-entropy. For $d\geq2$ we retain essential freeness as a separate
hypothesis because Proposition~\ref{prop:dooley-golodets-lattice-cpe} assumes
freeness; its possible redundancy is not needed here.

For the standard lattice we write $T^kx:=k.x$.
Section~\ref{sec:perfect-past} applies the Rohlin--Sinai--Kami\'nski theorem
on perfect $\sigma$-algebras, whose remote tail is the Pinsker algebra; CPE is
used there only to make that tail trivial. No quantitative entropy estimate is
used.

\subsection{Complex Gaussian laws and convergence in distribution}
\label{subsec:complex-gaussian}

A complex random vector $Z=(Z_1,\ldots,Z_m)$ is called centered Gaussian if
the real vector
\[
    (\Re Z_1,\Im Z_1,\ldots,\Re Z_m,\Im Z_m)
\]
is centered Gaussian. Its covariance matrix is
\begin{equation}
\label{eq:complex-covariance}
    C_{ab}:=\E[Z_a\overline{Z_b}].
\end{equation}
We call $Z$ \emph{proper} if
\begin{equation}
\label{eq:complex-properness}
    \E[Z_aZ_b]=0,
    \qquad 1\leq a,b\leq m.
\end{equation}
For a centered complex Gaussian vector this is equivalent to circular
symmetry. We write $Z\sim\mathcal{CN}_m(0,C)$ for the proper centered complex
Gaussian law with covariance $C$.

If $X_R$ and $X$ are Borel random elements of a separable metric space $S$,
then
\[
    X_R\dto X \quad\text{in }S
\]
means convergence in distribution, that is, weak convergence of the Borel
probability laws on $S$. Equivalently,
$\E f(X_R)\to\E f(X)$ for every bounded continuous $f:S\to\bR$.

In the scalar case, $Z\sim\mathcal{CN}(0,\sigma^2)$ means that
$\E|Z|^2=\sigma^2$. Because $\mathcal{CN}$ denotes the proper Gaussian law,
this is equivalent to saying that the real and imaginary parts are independent
$N(0,\sigma^2/2)$ variables. Consequently,
\begin{equation}
\label{eq:complex-gaussian-exponential}
    Z\sim\mathcal{CN}(0,\sigma^2)
    \quad\Longrightarrow\quad
    |Z|^2\sim\sigma^2\operatorname{Exp}(1).
\end{equation}
For $a,\vartheta>0$, we write $\Gamma(a,\vartheta)$ for the gamma distribution
with shape $a$ and scale $\vartheta$, whose density on $(0,\infty)$ is
\begin{equation}
\label{eq:gamma-shape-scale-density}
    x\longmapsto
    \frac{x^{a-1}e^{-x/\vartheta}}
         {\Gamma(a)\vartheta^a}.
\end{equation}
Here $\Gamma(a)$ in the denominator is Euler's gamma function. Its mean is $a\vartheta$ and its variance is $a\vartheta^2$; moreover, if
$X\sim\Gamma(a,1)$, then $\vartheta X\sim\Gamma(a,\vartheta)$. Thus
$\operatorname{Exp}(1)=\Gamma(1,1)$. If $\sigma^2>0$ and $Z_1,\ldots,Z_q$ are independent
$\mathcal{CN}(0,\sigma^2)$ variables, then
\begin{equation}
\label{eq:complex-gaussian-gamma}
    \frac1{\sigma^2}\sum_{j=1}^q|Z_j|^2\sim\Gamma(q,1).
\end{equation}
 \section{Completely positive entropy and perfect \texorpdfstring{$\sigma$}{sigma}-algebras}
\label{sec:perfect-past}

We extract from CPE the two structural facts used later. For the standard-lattice subaction
$T=(T^k)_{k\in\bZ^d}$, with $T^kx:=k.x$, we prove
\begin{equation}
\label{eq:section3-goal-decomposition}
    L^2_0(X,\mu)=\bigoplus_{k\in\bZ^d}U^kK_0,
\end{equation}
where the summands are martingale-difference spaces for the lexicographic
past. We then deduce that every centered Koopman spectral measure of $T$ is
absolutely continuous with respect to Lebesgue measure on $\bT^d$. The first
statement is obtained from the Rohlin--Sinai perfect-past theorem when $d=1$
and Kami\'nski's multidimensional extension when $d\geq2$.

\subsection{Perfect \texorpdfstring{$\sigma$}{sigma}-algebras and lexicographic cuts}
\label{subsec:perfect-lexicographic}

Equip $\bZ^d$ with the lexicographic order. For distinct
$k,\ell\in\bZ^d$, one has $k<_{\rm lex}\ell$ when $k_j<\ell_j$ at the first
coordinate where they differ. This is a translation-invariant total order:
every two distinct elements are comparable, and
\[
    k<_{\rm lex}\ell
    \quad\Longrightarrow\quad
    k+m<_{\rm lex}\ell+m
    \qquad(m\in\bZ^d).
\]
Its discrete feature that will be used repeatedly is that $k-e_d$ is the
immediate predecessor of $k$.

Following Kami\'nski, a \emph{cut} is an ordered pair $(L,R)$ of nonempty
subsets of $\bZ^d$ such that
\[
    \bZ^d=L\sqcup R,
    \qquad
    \ell<_{\rm lex}r
    \quad\text{for every }\ell\in L,\ r\in R.
\]
Thus $L$ and $R$ are the left and right sides of the cut. The cut is
\emph{principal} if $R$ has a least element. Equivalently, $L$ has a greatest
element; if the least element of $R$ is $k$, then the greatest element of $L$
is $k-e_d$. A cut is \emph{nonprincipal} if neither endpoint exists. In
Kami\'nski's terminology these nonprincipal cuts are the \emph{gaps}. They
occur only when $d\geq2$.

For example, in $\bZ^2$ the cut
\[
    L=\{\ell:\ell<_{\rm lex}0\},
    \qquad
    R=\{\ell:\ell\geq_{\rm lex}0\}
\]
is principal, with endpoints $-e_2$ and $0$. By contrast,
\[
    L=\{(m,n):m<0\},
    \qquad
    R=\{(m,n):m\geq0\}
\]
is nonprincipal: the left side has no greatest element and the right side has
no least element. For $d=1$ every cut is principal.

For a sub-$\sigma$-algebra $\mathcal A\subset\mathcal F$, write
$T^k\mathcal A$ for its image under $T^k$. With the Koopman convention
$U^kF=F\circ T^{-k}$ from Section~\ref{subsec:translation-spectra},
\begin{equation}
\label{eq:sigma-algebra-koopman-convention}
    U^kL^2(\mathcal A)=L^2(T^k\mathcal A).
\end{equation}
All $\sigma$-algebras are understood modulo null sets.

For $d=1$ the required statement is the Rohlin--Sinai perfect-$\sigma$-algebra
theorem \cite{RS61}. For $d\geq2$, Kami\'nski developed the corresponding
invariant-partition theory in \cite{Kam81}; we use the $\sigma$-algebra
formulation of \cite[Theorem~B, pp.~263--264]{Kam91}, where continuity at lexicographic
gaps is explicit. A later survey \cite{Kam96} describes perfect and strongly
invariant $\sigma$-algebras and their applications to Kolmogorov
$\bZ^d$-actions and spectral theory. The entropy identities in these results
are not used, and the existence statement from \cite[Theorem~B]{Kam91} does
not require a finite-entropy hypothesis.

\begin{theorem}[Rohlin--Sinai--Kami\'nski]
\label{thm:kaminski-perfect-past}
Let $T$ be a probability-preserving $\bZ^d$-action on a standard probability
space, and let $\Pi(T)$ denote its Pinsker $\sigma$-algebra. There is a
sub-$\sigma$-algebra $\mathcal A_0\subset\mathcal F$ such that, with
$\mathcal A_k:=T^k\mathcal A_0$,
\begin{align}
\label{eq:perfect-monotonicity}
    k<_{\rm lex}\ell &\quad\Longrightarrow\quad
      \mathcal A_k\subset\mathcal A_\ell,\\
\label{eq:perfect-exhaustion}
    \bigvee_{k\in\bZ^d}\mathcal A_k&=\mathcal F,\\
\label{eq:perfect-tail}
    \bigcap_{k\in\bZ^d}\mathcal A_k&=\Pi(T).
\end{align}
For every nonprincipal cut $(L,R)$,
\begin{equation}
\label{eq:perfect-gap-continuity}
    \bigvee_{k\in L}\mathcal A_k
    =\bigcap_{k\in R}\mathcal A_k.
\end{equation}
For $d=1$ the last condition is vacuous.
\end{theorem}

For $d\geq2$, condition $(a_1)$ of \cite[Theorem~B, pp.~263--264]{Kam91} states
$T^g\mathcal A_0\subset\mathcal A_0$ for $g<_{\rm lex}0$; translating this
inclusion gives \eqref{eq:perfect-monotonicity}. Conditions $(b_1)$ and
$(c_1)$ give \eqref{eq:perfect-exhaustion} and \eqref{eq:perfect-tail},
respectively, while condition $(e_1)$, imposed for every gap, gives
\eqref{eq:perfect-gap-continuity}. Such an $\mathcal A_0$ is called a
\emph{perfect $\sigma$-algebra}.

Under CPE, $\Pi(T)$ is trivial. By
Proposition~\ref{prop:dooley-golodets-lattice-cpe}, this applies to the
standard-lattice subaction of every essentially free CPE $\bR^d$-action.

\subsection{The wandering decomposition}
\label{subsec:wandering-decomposition}

Put $\mathscr H_k:=L^2(\mathcal A_k)$. Since $k-e_d$ is the immediate
predecessor of $k$, monotonicity gives
\begin{equation}
\label{eq:strict-lex-past}
    \bigvee_{\ell<_{\rm lex}k}\mathcal A_\ell
    =\mathcal A_{k-e_d}.
\end{equation}
Thus $\mathcal A_{-e_d}$ is the entire strict lexicographic past of the
origin. Define
\begin{equation}
\label{eq:innovation-space}
    K_0:=L^2(\mathcal A_0)\ominus L^2(\mathcal A_{-e_d}).
\end{equation}
It is the part of $L^2(\mathcal A_0)$ orthogonal to the strict past. By
\eqref{eq:sigma-algebra-koopman-convention},
\begin{equation}
\label{eq:innovation-translates}
    U^kK_0
    =L^2(\mathcal A_k)\ominus L^2(\mathcal A_{k-e_d}).
\end{equation}
For a nonprincipal cut $(L,R)$, equation
\eqref{eq:perfect-gap-continuity} gives
\[
    \overline{\bigcup_{k\in L}L^2(\mathcal A_k)}
    =\bigcap_{k\in R}L^2(\mathcal A_k).
\]
Thus the only successive orthogonal differences are those in
\eqref{eq:innovation-translates}. Lemma~\ref{lem:lexicographic-jump-decomposition}
in Appendix~\ref{app:lexicographic} gives the resulting orthogonal decomposition
used in Proposition~\ref{prop:wandering-innovation-decomposition}.

\begin{proposition}[Wandering decomposition]
\label{prop:wandering-innovation-decomposition}
Let $T$ be a probability-preserving $\bZ^d$-action and let
$\mathcal A_0$ be given by Theorem~\ref{thm:kaminski-perfect-past}. Then
\begin{equation}
\label{eq:cpe-wandering-decomposition}
    L^2(X,\mu)\ominus L^2(\Pi(T))
    =\bigoplus_{k\in\bZ^d}U^kK_0.
\end{equation}
In particular, if $T$ has CPE, then
\begin{equation}
\label{eq:cpe-wandering-decomposition-centered}
    L^2_0(X,\mu)=\bigoplus_{k\in\bZ^d}U^kK_0.
\end{equation}
\end{proposition}

\begin{proof}
For an increasing countable family of $\sigma$-algebras,
\[
    L^2\!\left(\bigvee_i\mathcal G_i\right)
    =\overline{\bigcup_iL^2(\mathcal G_i)},
\]
and for any countable family,
\[
    L^2\!\left(\bigcap_i\mathcal G_i\right)
    =\bigcap_iL^2(\mathcal G_i).
\]
Apply Lemma~\ref{lem:lexicographic-jump-decomposition} to
$\mathscr H_k=L^2(\mathcal A_k)$. Equations
\eqref{eq:perfect-exhaustion} and \eqref{eq:perfect-tail} identify the two
endpoint spaces, \eqref{eq:perfect-gap-continuity} supplies continuity at every
nonprincipal cut, and \eqref{eq:innovation-translates} identifies the
successive differences. If $T$ has CPE, then $L^2(\Pi(T))$ consists of the
constants.
\end{proof}

Let now $H$ be a separable complex Hilbert space. The standard isometric
identification
\[
    L^2(X,\mu;H)=L^2(X,\mu)\widehat\otimes H
\]
for the Bochner $L^2$-space is given, for example, in
\cite[Theorem~12.6.1]{Aub00}. Under this identification the constant
$H$-valued functions are $\bC\mathbf 1\widehat\otimes H$. Hence their
orthogonal complement is
\[
    L^2_0(X,\mu;H)
    :=\left\{Y\in L^2(X,\mu;H):\int_XY\,d\mu=0\right\}
    =L^2_0(X,\mu)\widehat\otimes H.
\]
The Koopman action on this space is
\[
    (U^kY)(x):=Y(T^{-k}x),
\]
which corresponds to $U^k\otimes I_H$ under the tensor-product
identification. Hence, for a CPE action, tensoring
\eqref{eq:cpe-wandering-decomposition-centered} with $H$ gives
\begin{equation}
\label{eq:hilbert-valued-wandering-decomposition}
    L^2_0(X,\mu;H)
    =\bigoplus_{k\in\bZ^d}U^k(K_0\widehat\otimes H).
\end{equation}

\subsection{Martingale differences and Lebesgue spectrum}
\label{subsec:martingale-spectral-consequences}

For a sub-$\sigma$-algebra $\mathcal G\subset\mathcal F$, the conditional
expectation of $Y\in L^2(X,\mu;H)$ is the orthogonal projection onto
$L^2(X,\mathcal G,\mu;H)$. Equivalently, for every $v\in H$,
\[
    \left\langle \E[Y\mid\mathcal G],v\right\rangle_H
    =\E\!\left[\langle Y,v\rangle_H\mid\mathcal G\right].
\]
If $D\in K_0\widehat\otimes H$, then
$D\perp L^2(\mathcal A_{-e_d};H)$. Using
\eqref{eq:strict-lex-past} and translating this orthogonality gives
\begin{equation}
\label{eq:lex-martingale-difference}
    \E\!\left[U^kD\mid
      \bigvee_{\ell<_{\rm lex}k}\mathcal A_\ell\right]=0.
\end{equation}
Thus the spaces in \eqref{eq:hilbert-valued-wandering-decomposition} are
martingale-difference spaces for the lexicographic filtration.

For the scalar decomposition, define
\[
    \mathcal W:\ell^2(\bZ^d;K_0)\longrightarrow L^2_0(X,\mu),
    \qquad
    \mathcal W((D_k)):=\sum_{k\in\bZ^d}U^kD_k.
\]
The series converges in $L^2$, and orthogonality gives
\[
    \|\mathcal W((D_k))\|_2^2
    =\sum_{k\in\bZ^d}\|D_k\|_2^2.
\]
By \eqref{eq:cpe-wandering-decomposition-centered}, $\mathcal W$ is unitary.
If $\lambda(j)$ denotes the regular shift
\[
    (\lambda(j)D)_k:=D_{k-j},
\]
then
\begin{equation}
\label{eq:regular-representation-unitary}
    U^j\mathcal W=\mathcal W\lambda(j),
    \qquad j\in\bZ^d.
\end{equation}
The Fourier transform in the lattice variable identifies $\lambda(j)$ with
multiplication by the character $\theta\mapsto
 e^{2\pi i\langle j,\theta\rangle}$ on $L^2(\bT^d;K_0)$. Therefore the
centered Koopman representation is a multiple of the regular representation,
and its spectral measures are absolutely continuous with respect to normalized
Lebesgue measure $d\theta$ on $\bT^d$.

\begin{corollary}[Lebesgue lattice spectrum]
\label{cor:cpe-lebesgue-lattice-spectrum}
For a CPE $\bZ^d$-action, every centered Koopman spectral measure is
absolutely continuous with respect to $d\theta$ on $\bT^d$.
\end{corollary}

For $d\geq2$, this spectral consequence is classical: see
\cite[Theorem~6]{Kam96} for $d=2$ and Kami\'nski--Liardet \cite{KL94} for
general $\bZ^d$. We include the derivation because Section~\ref{sec:lattice-clt}
uses the explicit martingale-difference decomposition, not merely its spectral
consequence.

\begin{corollary}[Absolute continuity of the Bartlett spectrum]
\label{cor:cpe-bartlett-ac}
Let $\eta$ be a stationary random measure on $\bR^d$ with local second
moments. If its translation action is essentially free and CPE, then
\begin{equation}
\label{eq:cpe-bartlett-density}
    d\sigma_\eta(\xi)=s_\eta(\xi)\,d\lambda_d(\xi)
\end{equation}
for some $s_\eta\in L^1_{\mathrm{loc}}(\bR^d)$, and
$s_\eta\lambda_d$ is translation bounded.
\end{corollary}

\begin{proof}
The standard-lattice subaction is CPE by
Proposition~\ref{prop:dooley-golodets-lattice-cpe}. Take
$f(x)=e^{-\pi\|x\|^2}$, whose Fourier transform is nowhere zero. Corollary
\ref{cor:cpe-lebesgue-lattice-spectrum} makes the lattice spectral measure of
$(\mathbb S f)^\circ$ absolutely continuous, and
Corollary~\ref{cor:lebesgue-lattice-spectrum-implies-bartlett-ac} then gives
$\sigma_\eta\ll\lambda_d$. Translation boundedness follows from
Proposition~\ref{prop:bartlett-spectrum}.
\end{proof}

Absolute continuity also follows directly from the continuous-action spectral
theorem of Dooley and Golodets \cite[Theorem~3.8]{DG12}, which gives infinite
Lebesgue spectrum for free CPE $\bR^d$-actions. The lattice derivation above
is retained because Section~\ref{sec:lattice-clt} uses the explicit
martingale-difference decomposition, not absolute continuity alone.
 
\section{A lattice Fourier central limit theorem}
\label{sec:lattice-clt}

Let $T=(T^k)_{k\in\bZ^d}$ be a CPE probability-preserving action on
$(X,\mathcal F,\mu)$ and let $H$ be a separable complex Hilbert space. We prove a Fourier central
limit theorem for $H$-valued lattice fields. Section~\ref{sec:perfect-past} gives
\begin{equation}
\label{eq:section4-wandering-decomposition}
    L^2_0(X,\mu;H)
    =\bigoplus_{r\in\bZ^d}U^r(K_0\widehat\otimes H).
\end{equation}
For a fixed field $Y\in L^2_0(X,\mu;H)$, we first use this decomposition to
replace $Y$, at almost every lattice frequency, by a single
martingale-difference variable. McLeish's theorem then gives the scalar
projections of the limit. A separate tightness argument upgrades these scalar
limits to convergence in distribution in the Hilbert space $H^m$. From this
point on, CPE enters only through \eqref{eq:section4-wandering-decomposition}
and the absolute continuity of centered lattice spectral measures from
Section~\ref{sec:perfect-past}.

For $h\in C_c^\infty(\bR^d;\bC)$ and $\theta\in\bT^d$, define
\begin{equation}
\label{eq:lattice-fourier-statistic}
    V_R(h,\theta)
    :=R^{-d/2}\sum_{k\in\bZ^d}
      h(k/R)e^{-2\pi i\langle k,\theta\rangle}U^kY.
\end{equation}

\subsection{Hilbert-valued Gaussian laws and weak convergence}
\label{subsec:hilbert-gaussian-weak-convergence}

Hilbert inner products are taken linear in the first variable. Let
$\mathcal H$ be a separable complex Hilbert space. A Borel random element
$G$ of $\mathcal H$ is a \emph{centered complex Gaussian random element} if,
for every $n\geq1$ and $u_1,\ldots,u_n\in\mathcal H$, the complex vector
\[
    \bigl(\langle G,u_1\rangle_{\mathcal H},\ldots,
          \langle G,u_n\rangle_{\mathcal H}\bigr)
\]
is centered Gaussian in the sense of Subsection~\ref{subsec:complex-gaussian}.
It is \emph{proper} if
\begin{equation}
\label{eq:hilbert-properness}
    \E\!\left[
      \langle G,u\rangle_{\mathcal H}
      \langle G,v\rangle_{\mathcal H}
    \right]=0,
    \qquad u,v\in\mathcal H.
\end{equation}
Its covariance operator $C$ is the positive trace-class operator determined by
\begin{equation}
\label{eq:hilbert-gaussian-covariance-operator}
    \langle Cv,u\rangle_{\mathcal H}
    =\E\!\left[
       \langle G,u\rangle_{\mathcal H}
       \overline{\langle G,v\rangle_{\mathcal H}}
      \right].
\end{equation}
Conversely, every positive trace-class operator $C$ determines a unique
proper complex Gaussian law on $\mathcal H$ with mean zero. Indeed, if
$Ce_n=\lambda_ne_n$ in an orthonormal eigenbasis and
$(\gamma_n)$ are independent $\mathcal{CN}(0,1)$ variables, then
$\sum_n\sqrt{\lambda_n}\gamma_ne_n$ converges in $L^2(\mathcal H)$ and has
that law. These definitions apply in particular to the product Hilbert space
$H^m$. When $H$ is infinite-dimensional, the limits below are therefore
Gaussian probability measures on an infinite-dimensional Hilbert space.

A family of $\mathcal H$-valued random elements $(X_R)$ is \emph{tight} if,
for every $\delta>0$, there is a compact set $K\subset\mathcal H$ such that
$\inf_R\Pp(X_R\in K)\geq1-\delta$.

We use two standard facts about weak convergence in a separable Hilbert space,
viewed as a real Hilbert space when applying real linear functionals. First,
if $(X_R)$ is tight and every continuous real linear functional of $X_R$
converges in distribution to the corresponding functional of $X$, then
$X_R\dto X$. Second, if $P_n$ are increasing finite-rank orthogonal
projections with $P_n\to I$ strongly, then
\begin{equation}
\label{eq:hilbert-tightness-criterion}
    \sup_R\E\|X_R\|_{\mathcal H}^2<\infty,
    \qquad
    \lim_{n\to\infty}\sup_R
       \E\|(I-P_n)X_R\|_{\mathcal H}^2=0
\end{equation}
imply tightness of the laws of $X_R$. For the second statement, choose a
subsequence of the projections for which the tail bounds are summable after
Markov's inequality. With arbitrarily high probability, the resulting tail
norms then tend to zero uniformly; together with a bound on the first
finite-dimensional projection, this confines $X_R$ to a relatively compact
subset of $\mathcal H$.

\subsection{Fourier transform of the martingale-difference coefficients}
\label{subsec:lattice-spectral-coordinates}

Put $\mathscr K:=K_0\widehat\otimes H$. By
\eqref{eq:section4-wandering-decomposition}, there is a unique family
$(D_r)_{r\in\bZ^d}\in\ell^2(\bZ^d;\mathscr K)$ such that
\begin{equation}
\label{eq:Y-innovation-expansion}
    Y=\sum_{r\in\bZ^d}U^rD_r,
    \qquad
    \sum_{r\in\bZ^d}\|D_r\|_{L^2(H)}^2=\E\|Y\|_H^2.
\end{equation}
The series converges in $L^2(X,\mu;H)$. Define
\begin{equation}
\label{eq:Gamma-Y-definition}
    \Gamma_Y(\theta)
    :=\sum_{r\in\bZ^d}e^{2\pi i\langle r,\theta\rangle}D_r,
    \qquad \theta\in\bT^d,
\end{equation}
where the Fourier series is interpreted in
$L^2(\bT^d;\mathscr K)$. Thus $\Gamma_Y$ is the Fourier transform of the
martingale-difference coefficient sequence $(D_r)$.

We fix the following torus version of the Lebesgue-point convention. For
$0<\eps<1/2$, let $B_\eps^{\bT}(\theta)$ be the ball of radius $\eps$ for
$\|\cdot\|_{\bT^d}$; its normalized Lebesgue measure is
$\lambda_d(B_\eps)$. If a representative of $q\in L^1(\bT^d)$ is fixed, we
call $\theta$ a \emph{Lebesgue point} of $q$ if
\begin{equation}
\label{eq:torus-lebesgue-point-definition}
    \frac{1}{\lambda_d(B_\eps)}
    \int_{B_\eps^{\bT}(\theta)}|q(t)-q(\theta)|\,dt
    \longrightarrow0
    \qquad(\eps\downarrow0).
\end{equation}
For a fixed representative of
$\Gamma\in L^2(\bT^d;\mathscr K)$, we call $\theta$ an
\emph{$L^2$-Lebesgue point} if
\begin{equation}
\label{eq:L2-Lebesgue-point-definition}
    \frac{1}{\lambda_d(B_\eps)}
    \int_{B_\eps^{\bT}(\theta)}
       \|\Gamma(t)-\Gamma(\theta)\|_{\mathscr K}^2\,dt
    \longrightarrow0.
\end{equation}
The Lebesgue differentiation theorem, applied to the separable Hilbert-valued
function $\Gamma$ and to $\|\Gamma\|_{\mathscr K}^2$, shows that almost every
$\theta$ is an $L^2$-Lebesgue point: expand
$\|\Gamma(t)-\Gamma(\theta)\|_{\mathscr K}^2$ and use differentiation of both
$\Gamma$ and $\|\Gamma\|_{\mathscr K}^2$ at $\theta$.

Choose a measurable representative of $\Gamma_Y$ and let $E_Y$ be its set of
$L^2$-Lebesgue points. Then $E_Y$ has full Lebesgue measure in $\bT^d$. Set
\begin{equation}
\label{eq:D-theta-definition}
    D_\theta:=\Gamma_Y(\theta)\in\mathscr K,
    \qquad \theta\in E_Y,
\end{equation}
and put $D_\theta:=0$ on $\bT^d\setminus E_Y$. Define the positive trace-class
operator $\mathcal F_Y(\theta)$ on $H$ by
\begin{equation}
\label{eq:operator-spectral-density}
    \langle \mathcal F_Y(\theta)v,u\rangle_H
    =\E\!\left[
       \langle D_\theta,u\rangle_H
       \overline{\langle D_\theta,v\rangle_H}
    \right],
    \qquad u,v\in H.
\end{equation}
Its trace is $\E\|D_\theta\|_H^2$. For $u,v\in H$, orthogonality of the
spaces $U^kK_0$ and Parseval give, for every $k\in\bZ^d$,
\begin{align}
\label{eq:scalar-cross-spectral-parseval}
    \E\!\left[
      \langle U^kY,u\rangle_H
      \overline{\langle Y,v\rangle_H}
    \right]
    &=\sum_{r\in\bZ^d}
      \E\!\left[
        \langle D_r,u\rangle_H
        \overline{\langle D_{r+k},v\rangle_H}
      \right] \\
    &=\int_{\bT^d}e^{2\pi i\langle k,\theta\rangle}
      \E\!\left[
        \langle\Gamma_Y(\theta),u\rangle_H
        \overline{\langle\Gamma_Y(\theta),v\rangle_H}
      \right]d\theta.
\end{align}
Hence the scalar lattice cross-spectral density of
$\langle Y,u\rangle_H$ and $\langle Y,v\rangle_H$ is, for almost every
$\theta$, the right-hand side of \eqref{eq:operator-spectral-density}.
Taking an orthonormal basis of $H$ and summing the diagonal identities gives
\begin{equation}
\label{eq:operator-density-trace-integral}
    \int_{\bT^d}\tr\mathcal F_Y(\theta)\,d\theta
    =\E\|Y\|_H^2.
\end{equation}

\begin{theorem}[Lattice Fourier CLT]
\label{thm:weighted-lattice-fourier-clt}
For every $\theta\in E_Y$ with $2\theta\neq0$ and every
$h_1,\ldots,h_m\in C_c^\infty(\bR^d;\bC)$, there is an $H^m$-valued complex Gaussian random element with mean zero
\[
    G_\theta
    =\bigl(G_\theta(h_1),\ldots,G_\theta(h_m)\bigr)
\]
such that
\begin{equation}
\label{eq:weighted-lattice-clt-joint}
    \bigl(V_R(h_1,\theta),\ldots,V_R(h_m,\theta)\bigr)
    \dto G_\theta
    \qquad\text{in }H^m.
\end{equation}
Its covariance is determined by
\begin{equation}
\label{eq:lattice-limit-covariance}
    \E\!\left[
      \langle G_\theta(h_a),u\rangle_H
      \overline{\langle G_\theta(h_b),v\rangle_H}
    \right]
    =\left(\int_{\bR^d}h_a\overline{h_b}\right)
      \langle\mathcal F_Y(\theta)v,u\rangle_H,
\end{equation}
for $u,v\in H$ and $1\leq a,b\leq m$, and
\begin{equation}
\label{eq:lattice-limit-properness}
    \E\!\left[
      \langle G_\theta(h_a),u\rangle_H
      \langle G_\theta(h_b),v\rangle_H
    \right]=0
\end{equation}
for all such $a,b,u,v$.
\end{theorem}

\subsection{Approximate identity and martingale approximation}
\label{subsec:lattice-martingale-approximation}

For $h\in C_c^\infty(\bR^d;\bC)$, define on $\bT^d$
\begin{equation}
\label{eq:lattice-kernel-definition}
    A_R^h(s)
    :=R^{-d/2}\sum_{k\in\bZ^d}h(k/R)e^{2\pi i\langle k,s\rangle},
    \qquad K_R^h(s):=|A_R^h(s)|^2.
\end{equation}
\begin{lemma}[Smooth lattice approximate identity]
\label{lem:smooth-lattice-approximate-identity}
For every $h\in C_c^\infty(\bR^d;\bC)$ and every $N\geq1$, there is
$C_{h,N}<\infty$ such that, for $R\geq1$ and $s\in\bT^d$,
\begin{equation}
\label{eq:lattice-kernel-decay}
    K_R^h(s)
    \leq C_{h,N}R^d
      \bigl(1+R\|s\|_{\bT^d}\bigr)^{-N}.
\end{equation}
Moreover,
\begin{equation}
\label{eq:lattice-kernel-mass}
    \int_{\bT^d}K_R^h(s)\,ds
    =R^{-d}\sum_{k\in\bZ^d}|h(k/R)|^2
    \longrightarrow\|h\|_2^2.
\end{equation}
Consequently, if $q\in L^1(\bT^d)$, the chosen representative satisfies
$q(0)=0$, and $0$ is a Lebesgue point in the sense of
\eqref{eq:torus-lebesgue-point-definition}, then
\begin{equation}
\label{eq:lattice-approximate-identity-conclusion}
    \int_{\bT^d}K_R^h(s)q(s)\,ds\longrightarrow0.
\end{equation}
\end{lemma}

\begin{proof}
Poisson summation, with the Fourier convention
\eqref{eq:fourier-convention}, gives
\begin{equation}
\label{eq:poisson-lattice-kernel}
    A_R^h(s)
    =R^{d/2}\sum_{n\in\bZ^d}\hat h\bigl(R(n-s)\bigr).
\end{equation}
Since $\hat h$ is Schwartz, summing its decay over $n\in\bZ^d$ yields
\eqref{eq:lattice-kernel-decay}, after increasing the decay exponent before
squaring. Orthogonality of the characters on $\bT^d$ gives the identity in
\eqref{eq:lattice-kernel-mass}, and the limit is the usual Riemann sum.

For \eqref{eq:lattice-approximate-identity-conclusion}, fix $N>d$ and
$\varepsilon>0$. Since $0$ is a Lebesgue point of $q$ and $q(0)=0$, choose
$\delta>0$ so that
\[
    \int_{B_r^{\bT}(0)}|q(s)|\,ds
    \leq \varepsilon\,\lambda_d(B_r)
    \qquad(0<r\leq\delta).
\]
The decay estimate \eqref{eq:lattice-kernel-decay} gives an
$O(\varepsilon)$ contribution on $B_{1/R}^{\bT}(0)$. Decomposing the rest of
$B_\delta^{\bT}(0)$ into the dyadic annuli
$2^j/R<\|s\|_{\bT^d}\leq2^{j+1}/R$, their contributions are bounded by
$C_{h,N}\varepsilon 2^{-j(N-d)}$, hence sum to $O(\varepsilon)$. On
$\{\|s\|_{\bT^d}>\delta\}$, the same decay estimate gives
$O(R^{d-N})\|q\|_1=o(1)$. Since $\varepsilon$ is arbitrary,
\eqref{eq:lattice-approximate-identity-conclusion} follows.
\end{proof}

For $\theta\in E_Y$, set
\begin{equation}
\label{eq:martingale-approximant-definition}
    V_R^D(h,\theta)
    :=R^{-d/2}\sum_{k\in\bZ^d}
      h(k/R)e^{-2\pi i\langle k,\theta\rangle}U^kD_\theta.
\end{equation}
The element $D_\theta$ belongs to the martingale-difference space
$\mathscr K=K_0\widehat\otimes H$. Hence the summands in
\eqref{eq:martingale-approximant-definition} are martingale differences in
lexicographic order.

\begin{proposition}[Frequency-dependent martingale approximation]
\label{prop:frequency-martingale-approximation}
For every $\theta\in E_Y$ and every
$h\in C_c^\infty(\bR^d;\bC)$,
\begin{equation}
\label{eq:martingale-approximation-L2}
    \E\|V_R(h,\theta)-V_R^D(h,\theta)\|_H^2
    \longrightarrow0.
\end{equation}
\end{proposition}

\begin{proof}
Under the unitary regular-representation identification from
Section~\ref{sec:perfect-past}, $Y$ corresponds to $\Gamma_Y$, while the
single martingale difference $D_\theta$ corresponds to the constant function
$t\mapsto D_\theta$. Hence $Y-D_\theta$ corresponds to
$t\mapsto\Gamma_Y(t)-\Gamma_Y(\theta)$. The coefficient sequence in
\eqref{eq:lattice-fourier-statistic} has Fourier transform
$A_R^h(t-\theta)$. The spectral isometry therefore gives
\begin{equation}
\label{eq:martingale-approximation-isometry}
    \E\|V_R(h,\theta)-V_R^D(h,\theta)\|_H^2
    =\int_{\bT^d}K_R^h(t-\theta)
      \|\Gamma_Y(t)-\Gamma_Y(\theta)\|_{\mathscr K}^2\,dt.
\end{equation}
By the definition of $E_Y$, the function
\[
    q_\theta(s)
    :=\|\Gamma_Y(\theta+s)-\Gamma_Y(\theta)\|_{\mathscr K}^2
\]
has $q_\theta(0)=0$, and $0$ is a Lebesgue point of $q_\theta$ in the sense of
\eqref{eq:torus-lebesgue-point-definition}. After the change of variables
$s=t-\theta$, the conclusion follows from
Lemma~\ref{lem:smooth-lattice-approximate-identity}.
\end{proof}

\subsection{Oscillatory ergodic averages and the martingale CLT}
\label{subsec:weighted-ergodic-martingale-clt}

To apply McLeish's theorem we must identify the limit of the sum of squared
martingale differences. The following weighted ergodic lemma treats both the
nonoscillatory terms and the oscillatory terms that occur in this calculation.

\begin{lemma}[Weighted oscillatory ergodic averages]
\label{lem:weighted-oscillatory-ergodic}
Under the standing CPE hypothesis, let $q\in L^1(X,\mu)$ and
$w\in C_c^\infty(\bR^d;\bC)$. For
$\alpha\in\bT^d$, put
\begin{equation}
\label{eq:weighted-ergodic-average}
    B_R(\alpha;w,q)
    :=R^{-d}\sum_{k\in\bZ^d}
      w(k/R)e^{2\pi i\langle k,\alpha\rangle}U^kq.
\end{equation}
Then, in $L^1(X,\mu)$,
\begin{equation}
\label{eq:weighted-ergodic-limit}
    B_R(\alpha;w,q)
    \longrightarrow
    \begin{cases}
      \displaystyle \left(\int_{\bR^d}w(x)\,dx\right)\E q,
          & \alpha=0,\\[6pt]
      0, & \alpha\neq0.
    \end{cases}
\end{equation}
\end{lemma}

\begin{proof}
Assume first that $q\in L^2$. Put
\begin{equation}
\label{eq:weighted-character-sum}
    b_R(s):=R^{-d}\sum_{k\in\bZ^d}
      w(k/R)e^{2\pi i\langle k,s\rangle}.
\end{equation}
The Riemann-sum bound gives $\sup_{R,s}|b_R(s)|<\infty$. Moreover,
$b_R(0)\to\int w$, while Poisson summation and the rapid decay of $\hat w$
show that $b_R(s)\to0$ for every $s\neq0$ in $\bT^d$. By the spectral
theorem,
\begin{equation}
\label{eq:weighted-ergodic-spectral-norm}
    \|B_R(\alpha;w,q)\|_2^2
    =\int_{\bT^d}|b_R(t+\alpha)|^2\,
      d\varsigma_q^{\mathrm{lat}}(t).
\end{equation}
Write $q=\E q+q^\circ$. The lattice spectral measure of $q^\circ$ is
absolutely continuous by Corollary~\ref{cor:cpe-lebesgue-lattice-spectrum}.
Here CPE is used only through absolute continuity, which rules out an atom at
the exceptional point $t=-\alpha$ for $b_R(t+\alpha)$. If $\alpha\neq0$,
dominated convergence in
\eqref{eq:weighted-ergodic-spectral-norm} gives
$B_R(\alpha;w,q^\circ)\to0$ in $L^2$, while the constant part equals
$b_R(\alpha)\E q\to0$. If $\alpha=0$, the same argument gives
$B_R(0;w,q^\circ)\to0$, and the constant part converges to
$(\int w)\E q$.

For general $q\in L^1$, choose bounded truncations $q_M\to q$ in $L^1$.
Since
\begin{equation}
\label{eq:weighted-ergodic-truncation-bound}
    \|B_R(\alpha;w,q-q_M)\|_1
    \leq R^{-d}\sum_k|w(k/R)|\,\|q-q_M\|_1,
\end{equation}
the error is bounded uniformly in $R$ by a constant times
$\|q-q_M\|_1$. Apply the $L^2$ result to $q_M$ and then let $M\to\infty$.
\end{proof}

We apply McLeish's martingale-array central limit theorem in the following
form \cite[Theorem~2.3]{McL74}.

\begin{lemma}[McLeish]
\label{lem:mcleish}
For each $n$, let $(X_{n,j})_{1\leq j\leq m_n}$ be a finite row of real,
square-integrable martingale differences. Suppose that, as $n\to\infty$,
\begin{align}
\label{eq:mcleish-max-prob}
    \max_{j\leq m_n}|X_{n,j}|&\longrightarrow0
      \quad\text{in probability},\\
\label{eq:mcleish-max-L2}
    \sup_n\E\max_{j\leq m_n}|X_{n,j}|^2&<\infty,\\
\label{eq:mcleish-square-sum}
    \sum_{j=1}^{m_n}X_{n,j}^2&\longrightarrow\sigma^2
      \quad\text{in probability},
\end{align}
where $\sigma^2>0$ is constant. Then
$\sum_jX_{n,j}\dto N(0,\sigma^2)$.
\end{lemma}

\begin{proposition}[Martingale Fourier CLT]
\label{prop:innovation-fourier-clt}
Let $D\in\mathscr K$ and let $\theta\in\bT^d$ satisfy $2\theta\neq0$. For
$h_1,\ldots,h_m\in C_c^\infty(\bR^d;\bC)$, set
\begin{equation}
\label{eq:innovation-fourier-statistic}
    M_R(h_a,\theta)
    :=R^{-d/2}\sum_{k\in\bZ^d}
      h_a(k/R)e^{-2\pi i\langle k,\theta\rangle}U^kD.
\end{equation}
Then there is an $H^m$-valued complex Gaussian random element with mean zero
\[
    G=\bigl(G(h_1),\ldots,G(h_m)\bigr)
\]
such that
\begin{equation}
\label{eq:innovation-fourier-convergence}
    (M_R(h_1,\theta),\ldots,M_R(h_m,\theta))\dto G
    \qquad\text{in }H^m.
\end{equation}
Its covariance is
\begin{equation}
\label{eq:innovation-fourier-covariance}
    \E\!\left[
      \langle G(h_a),u\rangle_H
      \overline{\langle G(h_b),v\rangle_H}
    \right]
    =\left(\int h_a\overline{h_b}\right)
      \E\!\left[
        \langle D,u\rangle_H
        \overline{\langle D,v\rangle_H}
      \right],
\end{equation}
and
\begin{equation}
\label{eq:innovation-fourier-properness}
    \E\!\left[
      \langle G(h_a),u\rangle_H
      \langle G(h_b),v\rangle_H
    \right]=0,
    \qquad u,v\in H.
\end{equation}
The covariance operator on $H^m$ determined by
\eqref{eq:innovation-fourier-covariance} is positive and trace class; its trace
is
\[
    \left(\sum_{a=1}^m\|h_a\|_2^2\right)\E\|D\|_H^2.
\]
Thus the Gaussian law in the statement is well defined by
Subsection~\ref{subsec:hilbert-gaussian-weak-convergence}.
\end{proposition}

\begin{proof}
Convergence as $R\to\infty$ can be checked along an arbitrary sequence
$R_n\to\infty$. Fix such a sequence when applying McLeish's array theorem
and suppress the index $n$ from the notation. We first prove convergence of
continuous real linear functionals on $H^m$. Every such functional has the form
$x\mapsto\Re\sum_{a=1}^m\langle x_a,u_a\rangle_H$ for suitable
$u_1,\ldots,u_m\in H$. Put $d_a:=\langle D,u_a\rangle_H$ and consider
\begin{equation}
\label{eq:cramer-wold-real-projection}
    L_R:=\Re\sum_{a=1}^m\langle M_R(h_a,\theta),u_a\rangle_H.
\end{equation}
After lexicographically ordering the finitely many active indices as
$k_1<_{\rm lex}\cdots<_{\rm lex}k_{m_R}$, write
$L_R=\sum_{j=1}^{m_R}X_{R,k_j}$ with
\begin{equation}
\label{eq:martingale-array-increments}
    X_{R,k}
    :=R^{-d/2}\Re z_{R,k},
    \qquad
    z_{R,k}:=e^{-2\pi i\langle k,\theta\rangle}
       \sum_{a=1}^m h_a(k/R)U^kd_a.
\end{equation}
Since $D\in\mathscr K=K_0\widehat\otimes H$, contraction against
$u_a\in H$ maps $\mathscr K$ into $K_0$; hence each
$d_a=\langle D,u_a\rangle_H$ belongs to $K_0$. Define
$\mathscr F_{R,0}:=\mathcal A_{k_1-e_d}$ and
$\mathscr F_{R,j}:=\mathcal A_{k_j}$ for $1\leq j\leq m_R$. By
\eqref{eq:strict-lex-past},
$\E[U^{k_j}d_a\mid\mathcal A_{k_j-e_d}]=0$, while
$\mathscr F_{R,j-1}\subset\mathcal A_{k_j-e_d}$. Hence
$(X_{R,k_j})_{j=1}^{m_R}$ is a martingale-difference row for
$(\mathscr F_{R,j})_{j=0}^{m_R}$.

The maximal conditions in Lemma~\ref{lem:mcleish} require only square
integrability. All active indices lie in a set of cardinality $O(R^d)$, and
for a fixed constant $C$,
\begin{equation}
\label{eq:martingale-max-domination}
    \max_k|X_{R,k}|
    \leq CR^{-d/2}\sum_{a=1}^m\max_{k}|U^kd_a|.
\end{equation}
For every $a$ and $c>0$, stationarity gives
\begin{equation}
\label{eq:martingale-max-tail}
    R^d\Pp\bigl(|d_a|>cR^{d/2}\bigr)
    \leq c^{-2}\E\!
      \left[|d_a|^2\ind_{\{|d_a|>cR^{d/2}\}}\right]
    \longrightarrow0.
\end{equation}
A union bound over the $O(R^d)$ active indices and the finitely many $a$,
together with \eqref{eq:martingale-max-tail}, gives
\eqref{eq:mcleish-max-prob}. Also
\begin{equation}
\label{eq:martingale-max-L2-bound}
    \E\max_k|X_{R,k}|^2
    \leq \E\sum_k|X_{R,k}|^2
    \leq C'R^{-d}\sum_{k,a}|h_a(k/R)|^2\E|d_a|^2,
\end{equation}
which is bounded uniformly in $R$. Hence
\eqref{eq:mcleish-max-L2} holds as well.

It remains to identify the sum of squares. The identity
$(\Re z)^2=(|z|^2+\Re z^2)/2$ gives
\begin{equation}
\label{eq:martingale-square-split}
    \sum_kX_{R,k}^2
    =\frac12R^{-d}\sum_k|z_{R,k}|^2
     +\frac12\Re\!\left(R^{-d}\sum_kz_{R,k}^2\right).
\end{equation}
Expanding the first term and applying
Lemma~\ref{lem:weighted-oscillatory-ergodic} with $\alpha=0$ gives
\begin{equation}
\label{eq:martingale-square-main-limit}
    R^{-d}\sum_k|z_{R,k}|^2
    \longrightarrow
    Q:=\sum_{a,b=1}^m
      \left(\int h_a\overline{h_b}\right)
      \E[d_a\overline{d_b}]
\end{equation}
in probability. Here $Q\geq0$, since it can also be written as
$\int\E|\sum_a h_a(x)d_a|^2\,dx$. For the second term,
\begin{equation}
\label{eq:martingale-square-oscillatory-term}
    R^{-d}\sum_kz_{R,k}^2
    =R^{-d}\sum_{k,a,b}
      h_a(k/R)h_b(k/R)e^{-4\pi i\langle k,\theta\rangle}
      U^k(d_ad_b),
\end{equation}
which converges to $0$ in $L^1$ by
Lemma~\ref{lem:weighted-oscillatory-ergodic}, since $2\theta\neq0$. Thus
\begin{equation}
\label{eq:martingale-realized-square-limit}
    \sum_kX_{R,k}^2\longrightarrow Q/2
    \qquad\text{in probability}.
\end{equation}
If $Q>0$, Lemma~\ref{lem:mcleish} gives
$L_R\dto N(0,Q/2)$. If $Q=0$, orthogonality of the martingale differences and
the same Riemann-sum calculation show that $\E|L_R|^2\to0$, so the conclusion
still holds with the degenerate normal law. Thus every continuous real linear
functional has the Gaussian limit prescribed by
\eqref{eq:innovation-fourier-covariance}. The same calculation applied to
$\Re(e^{i\phi}W_R)$, where
$W_R=\sum_a\langle M_R(h_a,\theta),u_a\rangle_H$, gives limiting variance
$Q/2$ for every $\phi$. Hence the limiting complex Gaussian variable $W$
satisfies $\E[W^2]=0$. Varying the vectors $u_a$ and polarizing gives
\eqref{eq:innovation-fourier-properness}.

It remains to verify tightness in $H^m$. Let $P_n$ be increasing finite-rank
orthogonal projections on $H$ with $P_n\to I$ strongly. Orthogonality of the translates
$U^k\mathscr K$ gives, for each $a$,
\begin{equation}
\label{eq:hilbert-tightness-tail}
    \E\|(I-P_n)M_R(h_a,\theta)\|_H^2
    =R^{-d}\sum_k|h_a(k/R)|^2
      \E\|(I-P_n)D\|_H^2.
\end{equation}
The Riemann sums are bounded uniformly in $R$, while the last expectation
tends to zero with $n$. The same orthogonality calculation, without
$I-P_n$, gives a uniform second-moment bound for every coordinate
$M_R(h_a,\theta)$. Thus \eqref{eq:hilbert-tightness-criterion}, applied with
the coordinatewise projections on $H^m$, gives tightness in $H^m$.
The weak-convergence criterion in
Subsection~\ref{subsec:hilbert-gaussian-weak-convergence} now upgrades the
convergence of all continuous real linear functionals to
\eqref{eq:innovation-fourier-convergence}.
\end{proof}

\begin{proof}[Proof of Theorem~\ref{thm:weighted-lattice-fourier-clt}]
Fix $\theta\in E_Y$ with $2\theta\neq0$. Proposition
\ref{prop:frequency-martingale-approximation} gives, simultaneously for the
finitely many functions,
\begin{equation}
\label{eq:joint-martingale-approximation}
    \max_{1\leq a\leq m}
    \|V_R(h_a,\theta)-V_R^D(h_a,\theta)\|_{L^2(H)}
    \longrightarrow0.
\end{equation}
Apply Proposition~\ref{prop:innovation-fourier-clt} with $D=D_\theta$.
Equation~\eqref{eq:operator-spectral-density} identifies its covariance with
\eqref{eq:lattice-limit-covariance}, and
\eqref{eq:joint-martingale-approximation} transfers the limit from the
martingale approximation to the original variables.
\end{proof}

\subsection{Several frequencies}
\label{subsec:lattice-theorem-joint-frequencies}

\begin{corollary}[Several frequencies]
\label{thm:lattice-several-frequencies}
Let $\theta_1,\ldots,\theta_q\in E_Y$ satisfy
\begin{equation}
\label{eq:joint-frequency-nonresonance}
    2\theta_a\neq0,
    \qquad
    \theta_a\neq\theta_b,
    \qquad
    \theta_a+\theta_b\neq0
    \quad (a\neq b).
\end{equation}
For each $a$, let
$h_{a,1},\ldots,h_{a,m_a}\in C_c^\infty(\bR^d;\bC)$. Then
\begin{equation}
\label{eq:lattice-several-frequencies-convergence}
    \bigl(V_R(h_{a,j},\theta_a)\bigr)_{a,j}
    \dto
    \bigl(G_{\theta_a}(h_{a,j})\bigr)_{a,j}
\end{equation}
in the finite product $\prod_{a=1}^q H^{m_a}$. The joint limit is complex Gaussian with mean zero. For each fixed $a$ its
covariance is given by
\eqref{eq:lattice-limit-covariance} with $\theta=\theta_a$. If $a\neq b$,
then for all $u,v\in H$ and all relevant $j,\ell$,
\begin{align}
\label{eq:lattice-distinct-frequency-covariance-zero}
    \E\!\left[
      \langle G_{\theta_a}(h_{a,j}),u\rangle_H
      \overline{\langle G_{\theta_b}(h_{b,\ell}),v\rangle_H}
    \right]&=0,\\
\label{eq:lattice-distinct-frequency-nonconjugate-zero}
    \E\!\left[
      \langle G_{\theta_a}(h_{a,j}),u\rangle_H
      \langle G_{\theta_b}(h_{b,\ell}),v\rangle_H
    \right]&=0.
\end{align}
Consequently, the random vectors associated with distinct frequencies are
independent.
\end{corollary}

\begin{proof}
Use the martingale approximation at each $\theta_a$. For a continuous real
linear functional of the resulting finite product, the analogue of
$z_{R,k}$ in \eqref{eq:martingale-array-increments} is a finite sum of terms
with phases $e^{-2\pi i\langle k,\theta_a\rangle}$. In the
$|z_{R,k}|^2$ term of \eqref{eq:martingale-square-split}, a cross term between
$\theta_a$ and $\theta_b$ carries phase
$e^{-2\pi i\langle k,\theta_a-\theta_b\rangle}$ and vanishes by
Lemma~\ref{lem:weighted-oscillatory-ergodic} when $a\neq b$. In the
$z_{R,k}^2$ term, the cross phase is
$e^{-2\pi i\langle k,\theta_a+\theta_b\rangle}$, which vanishes under
\eqref{eq:joint-frequency-nonresonance}. McLeish's theorem therefore gives
the stated Gaussian limits for all continuous real linear functionals.

The tightness estimate \eqref{eq:hilbert-tightness-tail} applies to every
coordinate, hence to the finite product. We obtain
\eqref{eq:lattice-several-frequencies-convergence} from the Hilbert-space
criterion in Subsection~\ref{subsec:hilbert-gaussian-weak-convergence}. The
same phase calculations give
\eqref{eq:lattice-distinct-frequency-covariance-zero} and
\eqref{eq:lattice-distinct-frequency-nonconjugate-zero}. Joint Gaussianity
then implies independence of the random vectors associated with distinct
frequencies.
\end{proof}

 \section{From lattice fields to stationary random measures}
\label{sec:continuum}

Let $\eta$ be a stationary random measure on $\bR^d$ with local second
moments whose translation action is essentially free and CPE. By
Corollary~\ref{cor:cpe-bartlett-ac},
$d\sigma_\eta=s_\eta\,d\lambda_d$ and $s_\eta\lambda_d$ is translation
bounded.

We deduce the continuum Fourier CLT from the lattice theorem of
Section~\ref{sec:lattice-clt}. We divide $\bR^d$ into unit
cubes and record the centered random measure in each cube as a Hilbert-valued
$\bZ^d$-field. The lattice theorem then gives a Gaussian limit for a statistic
in which $h$ is evaluated at the lattice points. We identify its variance with
$s_\eta(\xi)$ and show that the difference from $Z_R(\xi;h)$ converges to zero
in probability.

\begin{theorem}[Empirical Fourier CLT]
\label{thm:smooth-continuum-fourier-clt}
There is a $\lambda_d$-conull set $E_\eta\subset\bR^d$ such that, for every
$\xi\in E_\eta$, every $m\geq1$, and every
$h_1,\ldots,h_m\in C_c^\infty(\bR^d;\bC)$,
\begin{equation}
\label{eq:smooth-continuum-joint-clt}
    \bigl(Z_R(\xi;h_1),\ldots,Z_R(\xi;h_m)\bigr)
    \dto \bigl(G_\xi(h_1),\ldots,G_\xi(h_m)\bigr)
    \qquad\text{in }\bC^m,
\end{equation}
where the vector on the right is complex Gaussian with mean zero and
\begin{align}
\label{eq:smooth-continuum-limit-covariance}
    \E\!\left[G_\xi(h_a)\overline{G_\xi(h_b)}\right]
    &=s_\eta(\xi)
      \int_{\bR^d}h_a\overline{h_b}\,d\lambda_d,\\
\label{eq:smooth-continuum-limit-nonconjugate}
    \E\!\left[G_\xi(h_a)G_\xi(h_b)\right]&=0.
\end{align}
In particular,
\begin{equation}
\label{eq:smooth-continuum-scalar-clt}
    Z_R(\xi;h)
    \dto \mathcal{CN}\bigl(0,s_\eta(\xi)\|h\|_2^2\bigr).
\end{equation}
\end{theorem}

The class $C_c^\infty$ is used here because it supplies both the Fourier decay
needed in the lattice approximation and the derivative bounds needed in the
unit-cell Taylor expansion. We do not attempt to optimize the test-function
class.

\subsection{The unit-cell field}
\label{subsec:cell-field}

Let $C=(0,1)^d$, choose an integer $s>d/2$, and put
$\mathscr H_s:=H^s(C;\bC)^*$, the continuous complex-linear dual, equipped
with its standard Hilbert-space structure. The pairing
$\langle z,\phi\rangle:=z(\phi)$ between $\mathscr H_s$ and
$H^s(C;\bC)$ is complex-bilinear; it is distinct from the sesquilinear
Hilbert inner product $\langle\cdot,\cdot\rangle_{\mathscr H_s}$, which is
linear in the first variable. Since
$H^s(C)\hookrightarrow C(\overline C)$, every finite signed measure on $C$
defines an element of $\mathscr H_s$. Define
\begin{equation}
\label{eq:cell-field-definition}
    \langle Y(\omega),\phi\rangle
    :=\int_C\phi(u)\,dM_\omega(u),
    \qquad \phi\in H^s(C).
\end{equation}
Thus $Y$ records the centered random measure in the unit cube.

\begin{lemma}
\label{lem:cell-field-L2}
The random element $Y$ belongs to $L^2_0(\eta;\mathscr H_s)$ and, for
$k\in\bZ^d$,
\begin{equation}
\label{eq:translated-cell-field}
    \langle U^kY,\phi\rangle
    =\int_{C-k}\phi(x+k)\,dM(x).
\end{equation}
Moreover, for $\eta$-almost every $\omega$,
$M_\omega(\partial(C-k))=0$ for every $k\in\bZ^d$. For such $\omega$, the
measures $M_\omega|_{C-k}$ determine $M_\omega$.
\end{lemma}

\begin{proof}
For each $\phi\in H^s(C)$, the scalar map
$\omega\mapsto\langle Y(\omega),\phi\rangle$ is measurable. Since
$\mathscr H_s$ is a separable Hilbert space, Pettis measurability therefore
makes $Y$ a strongly measurable $\mathscr H_s$-valued map. Sobolev embedding
gives $\|Y\|_{\mathscr H_s}\leq c_s(\omega(C)+\rho_\eta)$, so local second moments
imply $Y\in L^2(\eta;\mathscr H_s)$. Stationarity gives $\E_\eta Y=0$, and
\eqref{eq:translated-cell-field} follows from the Koopman convention. Finally,
for every $k\in\bZ^d$,
\[
    \E_\eta\omega(\partial(C-k))
    =\rho_\eta\lambda_d(\partial C)=0.
\]
Hence $\omega(\partial(C-k))=0$ almost surely for each $k$. Since $\bZ^d$ is
countable, this holds for all $k$ on one set of full $\eta$-measure. The same
statement holds for $M_\omega$, because $\lambda_d(\partial C)=0$.
\end{proof}

Apply Section~\ref{sec:lattice-clt} to $Y$. We keep the notation $E_Y$ and
$D_\theta$ from Subsection~\ref{subsec:lattice-spectral-coordinates}.

\subsection{The variance at a continuum frequency}
\label{subsec:cell-dealiasing}

Fix a representative of $s_\eta$ for the pointwise formulas in this subsection and in the proof of Theorem~\ref{thm:smooth-continuum-fourier-clt}.

For $\xi\in\bR^d$, define the bounded linear functional
\begin{equation}
\label{eq:cell-functional-definition}
    \ell_\xi(z):=\langle z,\chi_\xi|_C\rangle,
    \qquad z\in\mathscr H_s.
\end{equation}
By the Riesz representation theorem there is a unique
$u_\xi\in\mathscr H_s$ such that
$\ell_\xi(z)=\langle z,u_\xi\rangle_{\mathscr H_s}$; thus the scalar
projections below use the Hilbert-space convention of Section~\ref{sec:lattice-clt}.
Put $\beta:=\widehat{\ind}_C$. Then
$\beta(0)=1$ and $\beta(m)=0$ for every
$m\in\bZ^d\setminus\{0\}$.

The next proposition identifies the variance when the lattice frequency is the
class of $\xi$ modulo $\bZ^d$. Its exceptional set is chosen independently of
$\xi$.

\begin{proposition}[Variance at a continuum frequency]
\label{prop:cell-periodization-dealiasing}
There is a set $E_Y^{\mathrm{cell}}\subset E_Y$ of full Lebesgue measure such that, for every
$\theta\in E_Y^{\mathrm{cell}}$ and every $\xi\in\bR^d$,
\begin{equation}
\label{eq:cell-periodization-formula}
    \E_\eta|\ell_\xi(D_\theta)|^2
    =\sum_{m\in\bZ^d}
       |\beta(\theta+m-\xi)|^2s_\eta(\theta+m).
\end{equation}
Consequently, if $\xi=\theta+q$ for some $q\in\bZ^d$, then
\begin{equation}
\label{eq:exact-dealiasing}
    \E_\eta|\ell_\xi(D_\theta)|^2=s_\eta(\xi).
\end{equation}
\end{proposition}

\begin{proof}
Fix $\xi$. By definition, $\ell_\xi(Y)=M(\ind_C\chi_\xi)$. By
\eqref{eq:bartlett-indicator-spectral-measure}, its spectral measure for the
$\bR^d$-action has density
\[
    |\beta(\lambda-\xi)|^2s_\eta(\lambda).
\]
Passing to the standard-lattice subaction sums this density over translates by
$\bZ^d$. Section~\ref{sec:lattice-clt} identifies the resulting lattice
spectral density at $\theta$ with $\E_\eta|\ell_\xi(D_\theta)|^2$. This proves
\eqref{eq:cell-periodization-formula} for almost every $\theta$ when $\xi$ is
fixed.

We now choose the exceptional set independently of $\xi$. The identity holds
simultaneously for $\xi$ in a fixed countable dense subset of $\bR^d$ outside
a null set of $\theta$'s. On every compact set $K\subset\bR^d$,
\begin{equation}
\label{eq:cell-fourier-product-decay}
    |\beta(\theta+m-\xi)|^2
    \leq C_K\prod_{j=1}^d(1+|m_j|)^{-2},
    \qquad \xi\in K.
\end{equation}
Let $w_m:=\prod_{j=1}^d(1+|m_j|)^{-2}$. Translation boundedness gives
\[
    \int_{\bT^d}\sum_{m\in\bZ^d}w_m s_\eta(\theta+m)\,d\theta
    <\infty.
\]
Hence, for almost every $\theta$,
$\sum_m w_m s_\eta(\theta+m)<\infty$. For such $\theta$,
\eqref{eq:cell-fourier-product-decay} and the Weierstrass $M$-test give local
uniform convergence in $\xi$ of the series on the right of
\eqref{eq:cell-periodization-formula}. The left side is also continuous in
$\xi$, because
$\xi\mapsto\chi_\xi|_C$ is continuous in $H^s(C)$. Thus the identity extends
from the dense subset to every $\xi$. Intersecting this conull set with $E_Y$
gives $E_Y^{\mathrm{cell}}$.

If $\xi=\theta+q$, then
$\beta(\theta+m-\xi)=\beta(m-q)$, which is $1$ for $m=q$ and $0$ otherwise.
Equation~\eqref{eq:exact-dealiasing} follows.
\end{proof}

\subsection{Approximation on unit cells}
\label{subsec:cell-freezing}

Fix $\xi\in\bR^d$ and let $\theta=[\xi]\in\bT^d$ be its class modulo
$\bZ^d$. If $\xi=\theta+q$ with $q\in\bZ^d$, then for $u\in C$ and
$k\in\bZ^d$,
\[
    \chi_\xi(u-k)=\chi_\xi(u)e^{-2\pi i\langle k,\xi\rangle}
    =\chi_\xi(u)e^{-2\pi i\langle k,\theta\rangle}.
\]
Lemma~\ref{lem:cell-field-L2} therefore gives
\begin{equation}
\label{eq:cell-decomposition-empirical-transform}
    Z_R(\xi;h)
    =R^{-d/2}\sum_{k\in\bZ^d}e^{-2\pi i\langle k,\theta\rangle}
      \left\langle
        U^kY,
        \chi_\xi(\cdot)h\bigl((\cdot-k)/R\bigr)
      \right\rangle.
\end{equation}
Write $\check h(x):=h(-x)$ and set
\begin{equation}
\label{eq:frozen-cell-statistic}
    \widetilde Z_R(\xi;h)
    :=\ell_\xi\bigl(V_R(\check h,\theta)\bigr)
    =R^{-d/2}\sum_{k\in\bZ^d}
       h(-k/R)e^{-2\pi i\langle k,\theta\rangle}
       \ell_\xi(U^kY).
\end{equation}
The following proposition shows that the variation of $h$ inside each unit
cube does not affect the limit.

\begin{proposition}[Approximation on unit cells]
\label{prop:cell-freezing}
If $\theta=[\xi]\in E_Y$ and $2\theta\neq0$, then, for every
$h\in C_c^\infty(\bR^d;\bC)$,
\begin{equation}
\label{eq:cell-freezing-conclusion}
    Z_R(\xi;h)-\widetilde Z_R(\xi;h)\pto0.
\end{equation}
The same conclusion holds simultaneously for any fixed finite family of such
functions.
\end{proposition}

\begin{proof}
Choose an integer $r\geq s$. For $u\in C$, Taylor's formula at $-k/R$ gives
\begin{equation}
\label{eq:cell-taylor-expansion}
    h\bigl((u-k)/R\bigr)
    =\sum_{|\alpha|\leq r}
       \frac{u^\alpha}{\alpha!R^{|\alpha|}}
       (\partial^\alpha h)(-k/R)
      +\mathcal R_{R,k}(u).
\end{equation}
The term $\alpha=0$ gives \eqref{eq:frozen-cell-statistic}. If
$|\alpha|\geq1$, put
\[
    h_\alpha(x):=(\partial^\alpha h)(-x),
    \qquad
    \ell_{\xi,\alpha}(z):=\langle z,u^\alpha\chi_\xi(u)\rangle.
\]
The corresponding contribution to
\eqref{eq:cell-decomposition-empirical-transform} is
\[
    \frac{R^{-|\alpha|}}{\alpha!}
      \ell_{\xi,\alpha}\bigl(V_R(h_\alpha,\theta)\bigr).
\]
By Theorem~\ref{thm:weighted-lattice-fourier-clt}, the random variables
$\ell_{\xi,\alpha}(V_R(h_\alpha,\theta))$ are tight. Since
$|\alpha|\geq1$, multiplication by $R^{-|\alpha|}$ makes each of these terms
converge to zero in probability.

For the remainder, let $|\gamma|\leq s$. Differentiating
\eqref{eq:cell-taylor-expansion} in $u$ leaves the Taylor remainder, at order
$r-|\gamma|$, for
$R^{-|\gamma|}(\partial^\gamma h)((u-k)/R)$. Hence every such derivative of
$\mathcal R_{R,k}$ is $O(R^{-(r+1)})$. Multiplication by the fixed smooth
function $\chi_\xi$ preserves this order. Thus, for the fixed frequency $\xi$,
\[
    \|\chi_\xi\mathcal R_{R,k}\|_{H^s(C)}
    \leq C_{h,\xi}R^{-(r+1)}
\]
uniformly over the cells meeting the support of $h(\cdot/R)$. There are
$O(R^d)$ such cells. Using the uniform $L^2(\eta;\mathscr H_s)$ bound for
$U^kY$ and Minkowski's inequality therefore gives
\begin{equation}
\label{eq:cell-remainder-L2}
    \|\mathrm{Rem}_R\|_{L^2(\eta)}
    =O\bigl(R^{d/2-r-1}\bigr)=o(1).
\end{equation}
This proves \eqref{eq:cell-freezing-conclusion}. The argument applies to each
member of a fixed finite family.
\end{proof}

\begin{proof}[Proof of Theorem~\ref{thm:smooth-continuum-fourier-clt}]
Let $E_\eta$ be the set of $\xi\in\bR^d$ such that $\xi$ is a Lebesgue point of this representative,
$[\xi]\in E_Y^{\mathrm{cell}}$, and $2[\xi]\neq0$. By the Lebesgue
differentiation theorem and Proposition~\ref{prop:cell-periodization-dealiasing},
$E_\eta$ is $\lambda_d$-conull. The Lebesgue-point condition is not needed
for the smooth CLT itself; it is included so that the same set $E_\eta$ can
be used for the sharp-window result in Section~\ref{sec:scattering}.

Fix $\xi\in E_\eta$ and put $\theta=[\xi]$. Theorem
\ref{thm:weighted-lattice-fourier-clt}, applied to
$\check h_1,\ldots,\check h_m$, gives joint convergence of the
$\mathscr H_s$-valued lattice sums. Applying the bounded linear functional
$\ell_\xi$ gives a complex Gaussian vector with mean zero. By
\eqref{eq:exact-dealiasing}, its covariance is
\[
    s_\eta(\xi)
    \int_{\bR^d}h_a\overline{h_b}\,d\lambda_d,
\]
and \eqref{eq:lattice-limit-properness} gives
$\E[G_\xi(h_a)G_\xi(h_b)]=0$. Proposition~\ref{prop:cell-freezing} transfers
this convergence to the variables $Z_R(\xi;h_a)$.
\end{proof}

Although a representative of $s_\eta$ was fixed in constructing $E_\eta$,
the theorem does not depend on that choice. Two representatives agree off a
$\lambda_d$-null set; after removing that set, the corresponding good sets
have the same pointwise variance $s_\eta(\xi)$.

Our choice of $E_\eta$ excludes the frequencies for which $2[\xi]=0$ in
$\bT^d$; this is a null set. In particular, $0\notin E_\eta$. The exclusion
comes from the lattice CLT used above and does not assert that the continuum
CLT fails at the other excluded frequencies.

\subsection{Several frequencies}
\label{subsec:smooth-continuum-clt}

\begin{corollary}[Several frequencies]
\label{cor:continuum-several-frequencies}
For every $q\geq1$ there is a $\lambda_{dq}$-conull set
$E_\eta^{(q)}\subset(\bR^d)^q$ such that the following holds. If
$(\xi_1,\ldots,\xi_q)\in E_\eta^{(q)}$ and, for each $a$,
$h_{a,1},\ldots,h_{a,m_a}\in C_c^\infty(\bR^d;\bC)$, then
\begin{equation}
\label{eq:continuum-several-frequencies-convergence}
    \bigl(Z_R(\xi_a;h_{a,j})\bigr)_{a,j}
    \dto
    \bigl(G_{\xi_a}(h_{a,j})\bigr)_{a,j}
\end{equation}
in $\bC^{m_1+\cdots+m_q}$. For $a\neq b$, the two Gaussian vectors
$(G_{\xi_a}(h_{a,j}))_j$ and $(G_{\xi_b}(h_{b,j}))_j$ are independent.
\end{corollary}

\begin{proof}
Take $E_\eta^{(q)}$ to consist of the tuples with every $\xi_a\in E_\eta$ and,
for $a\neq b$,
\[
    [\xi_a]\neq[\xi_b],
    \qquad
    [\xi_a]+[\xi_b]\neq0
    \quad\text{in }\bT^d.
\]
Its complement is null. Apply Corollary~\ref{thm:lattice-several-frequencies}
to the lattice frequencies $[\xi_a]$, apply the functionals $\ell_{\xi_a}$,
and then use Proposition~\ref{prop:cell-freezing} for each coordinate.
\end{proof}

The congruence exclusions in this corollary are sufficient nonresonance
conditions inherited from the fixed unit-lattice reduction. They form a null
set and are not asserted to be optimal continuum-frequency conditions.
 \section{Consequences for ball windows and scattering statistics}
\label{sec:scattering}

We derive ball-window and point-process consequences of
Theorem~\ref{thm:smooth-continuum-fourier-clt}. Smooth approximation gives the
ball-window CLT. For point processes, we compare the centered ball transform
with the raw Fourier sum and use ergodicity to replace $R^{d/2}$ by the square
root of the observed point count. Squaring the Gaussian limits gives the
exponential and gamma limits.

Throughout the section, $E_\eta$ denotes the conull set from
Section~\ref{sec:continuum}. Our choice of this set excludes $0$.

\subsection{Ball windows}
\label{subsec:ball-windows}

Put $b:=\ind_{B_1}$. With the Fourier convention of
Section~\ref{sec:preliminaries},
\begin{equation}
\label{eq:ball-fourier-bessel}
    \hat b(\xi)=\|\xi\|^{-d/2}J_{d/2}(2\pi\|\xi\|),
\end{equation}
with the value at the origin understood by continuity. In particular,
$|\hat b(\xi)|\leq C(1+\|\xi\|)^{-(d+1)/2}$.

For $\xi\in\bR^d$, define
\begin{equation}
\label{eq:centered-ball-transform}
    Z_R^{B}(\xi)
    :=R^{-d/2}M(\ind_{B_R}\chi_\xi)
    =Z_R(\xi;b).
\end{equation}
The second equality extends the notation in
\eqref{eq:empirical-fourier-definition} from smooth functions to the indicator
$b$.

\begin{theorem}[Fourier CLT for ball windows]
\label{thm:ball-fourier-clt}
Let $\eta$ be a stationary random measure with local second moments whose
translation action is essentially free and CPE. Then, for every
$\xi\in E_\eta$,
\begin{equation}
\label{eq:ball-fourier-clt}
    Z_R^{B}(\xi)
    \dto \mathcal{CN}\bigl(0,s_\eta(\xi)\lambda_d(B_1)\bigr).
\end{equation}
If $(\xi_1,\ldots,\xi_q)$ belongs to the conull set
$E_\eta^{(q)}$ from Corollary~\ref{cor:continuum-several-frequencies}, then the
ball transforms at these frequencies converge jointly, and their Gaussian
limits are independent.
\end{theorem}

\begin{proof}
Fix $\xi\in E_\eta$ and choose $h_n\in C_c^\infty(\bR^d)$ with
$h_n\to b$ in $L^2(\bR^d)$. For each fixed $n$,
Theorem~\ref{thm:smooth-continuum-fourier-clt} gives the limit of
$Z_R(\xi;h_n)$.

By the extension in Lemma~\ref{lem:bartlett-indicators}, the Bartlett
covariance identity applies to $b-h_n$. The decay in
\eqref{eq:ball-fourier-bessel}, together with the Schwartz decay of
$\hat h_n$, then allows Lemma~\ref{lem:differentiation-decaying-kernels} to be
applied after expanding $|\hat b-\hat h_n|^2$. Hence
\begin{equation}
\label{eq:ball-smooth-approximation-L2}
    \lim_{R\to\infty}
      \E_\eta\left|Z_R^{B}(\xi)-Z_R(\xi;h_n)\right|^2
    =s_\eta(\xi)\|b-h_n\|_2^2.
\end{equation}
The right-hand side tends to zero as $n\to\infty$. Moreover,
$\|h_n\|_2^2\to\|b\|_2^2=\lambda_d(B_1)$, so the Gaussian limits for the
smooth functions converge in distribution to
$\mathcal{CN}(0,s_\eta(\xi)\lambda_d(B_1))$. The standard
converging-together theorem \cite{Bil99} therefore gives
\eqref{eq:ball-fourier-clt}.

Now let $(\xi_1,\ldots,\xi_q)\in E_\eta^{(q)}$. For each fixed $n$,
Corollary~\ref{cor:continuum-several-frequencies} gives joint convergence of
\[
    \bigl(Z_R(\xi_1;h_n),\ldots,Z_R(\xi_q;h_n)\bigr)
\]
to independent complex Gaussian variables with variances
$s_\eta(\xi_a)\|h_n\|_2^2$. Equation
\eqref{eq:ball-smooth-approximation-L2}, applied separately at each $\xi_a$,
shows that the corresponding vector of ball transforms is approximated in
probability by this smooth vector. As $n\to\infty$, the joint Gaussian laws
converge to the product of the laws
$\mathcal{CN}(0,s_\eta(\xi_a)\lambda_d(B_1))$. A second application of the
converging-together theorem proves both the joint convergence and the stated
independence.
\end{proof}

\subsection{Point-process Fourier sums and scattering intensity}
\label{subsec:point-process-scattering}

Assume now that $\eta$ is a stationary point process with intensity
$\rho_\eta>0$. Since CPE implies ergodicity and
$F(\omega):=\omega(Q)$ belongs to $L^2(\eta)$ by local second moments, Wiener's
pointwise ergodic theorem for Euclidean ball averages \cite{Wie39}, with
$Q=[-1/2,1/2)^d$, gives
\[
    \frac{1}{\lambda_d(B_R)}\int_{B_R}F(t.\omega)\,dt
    \longrightarrow \E_\eta F=\rho_\eta
    \qquad\text{almost surely}.
\]
If $c=\sqrt d/2$, then
$F(t.\omega)=\omega(Q+t)$, and for $R>c$ Fubini's theorem gives
\[
    \int_{B_{R-c}}F(t.\omega)\,dt
    \leq \omega(B_R)
    \leq \int_{B_{R+c}}F(t.\omega)\,dt.
\]
Since $\lambda_d(B_{R\pm c})/\lambda_d(B_R)\to1$, it follows that
\begin{equation}
\label{eq:point-count-ergodic}
    \frac{\mathbb S\ind_{B_R}}{\lambda_d(B_R)}\longrightarrow\rho_\eta
    \qquad\text{almost surely}.
\end{equation}

Theorem~\ref{thm:ball-fourier-clt} concerns the centered random measure $M$,
whereas the scattering statistic uses the uncentered point sum. By the
definition of $M$, for every $\xi$,
\begin{equation}
\label{eq:point-process-centered-uncentered}
    \mathbb S(\ind_{B_R}\chi_\xi)
    =M(\ind_{B_R}\chi_\xi)
      +\rho_\eta\int_{B_R}\chi_\xi(x)\,d\lambda_d(x).
\end{equation}
For a fixed nonzero frequency, the second term is of smaller order than the
$R^{d/2}$ normalization. Indeed, \eqref{eq:ball-fourier-bessel} gives
\begin{equation}
\label{eq:ball-mean-negligible}
    R^{-d/2}\left|\int_{B_R}\chi_\xi(x)\,d\lambda_d(x)\right|
    =R^{d/2}|\hat b(R\xi)|=O_\xi(R^{-1/2}).
\end{equation}
Thus, for every $\xi\in E_\eta$,
\begin{equation}
\label{eq:point-process-centered-comparison}
    R^{-d/2}\mathbb S(\ind_{B_R}\chi_\xi)
    =Z_R^B(\xi)
      +\rho_\eta R^{-d/2}\int_{B_R}\chi_\xi\,d\lambda_d
    =Z_R^B(\xi)+o(1).
\end{equation}

\begin{corollary}[Empirical Fourier sums]
\label{cor:point-process-fourier-clt}
Under the hypotheses of Theorem~\ref{thm:ball-fourier-clt}, suppose that
$\eta$ is a point process with $\rho_\eta>0$. Then, for every
$\xi\in E_\eta$,
\begin{equation}
\label{eq:point-process-fourier-clt}
    \frac{1}{\sqrt{\mathbb S\ind_{B_R}}}
      \sum_{x\in\omega\cap B_R}\chi_\xi(x)
    \dto \mathcal{CN}\bigl(0,S_\eta(\xi)\bigr).
\end{equation}
The left-hand side is defined to be $0$ on
$\{\mathbb S\ind_{B_R}=0\}$. If $(\xi_1,\ldots,\xi_q)\in E_\eta^{(q)}$,
the corresponding Fourier sums
converge jointly, and their Gaussian limits are independent.
\end{corollary}

\begin{proof}
Since $\lambda_d(B_R)=R^d\lambda_d(B_1)$, equation
\eqref{eq:point-count-ergodic} implies
\begin{equation}
\label{eq:point-count-square-root}
    \sqrt{\frac{\mathbb S\ind_{B_R}}{R^d}}
    \longrightarrow
    \sqrt{\rho_\eta\lambda_d(B_1)}
    \qquad\text{almost surely}.
\end{equation}
Since the limit in \eqref{eq:point-count-square-root} is positive,
$\mathbb S\ind_{B_R}>0$ eventually for almost every configuration. For such
$R$,
\[
    \frac{\mathbb S(\ind_{B_R}\chi_\xi)}
         {\sqrt{\mathbb S\ind_{B_R}}}
    =
    \frac{R^{-d/2}\mathbb S(\ind_{B_R}\chi_\xi)}
         {\sqrt{\mathbb S\ind_{B_R}/R^d}}.
\]
By \eqref{eq:point-process-centered-comparison} and
Theorem~\ref{thm:ball-fourier-clt}, the numerator converges in distribution to
$\mathcal{CN}(0,s_\eta(\xi)\lambda_d(B_1))$. The denominator converges almost
surely to $\sqrt{\rho_\eta\lambda_d(B_1)}$. Slutsky's theorem therefore gives
\eqref{eq:point-process-fourier-clt}, since
$S_\eta=s_\eta/\rho_\eta$.

For $(\xi_1,\ldots,\xi_q)\in E_\eta^{(q)}$, use the joint ball-window limit in
Theorem~\ref{thm:ball-fourier-clt}, apply
\eqref{eq:point-process-centered-comparison} to each coordinate, and use the
same almost-sure denominator limit \eqref{eq:point-count-square-root}. The
multivariate form of Slutsky's theorem gives the joint assertion, and the
independence is the independence already present in the joint ball-window
limit.
\end{proof}

\begin{corollary}[Empirical scattering intensity]
\label{cor:empirical-scattering-exponential}
Under the hypotheses of Corollary~\ref{cor:point-process-fourier-clt}, for
$\xi\in E_\eta$ the empirical scattering intensity from
\eqref{eq:point-process-empirical-transform} satisfies
\begin{equation}
\label{eq:empirical-scattering-exponential}
    I_R(\xi)
    =\frac{1}{\mathbb S\ind_{B_R}}
      \left|\sum_{x\in\omega\cap B_R}\chi_\xi(x)\right|^2
    \dto S_\eta(\xi)\,\operatorname{Exp}(1).
\end{equation}
As in \eqref{eq:point-process-empirical-transform}, this quantity is defined
to be $0$ on $\{\mathbb S\ind_{B_R}=0\}$.
\end{corollary}

\begin{proof}
If $Z\sim\mathcal{CN}(0,S_\eta(\xi))$, then
$|Z|^2\sim S_\eta(\xi)\operatorname{Exp}(1)$ by
\eqref{eq:complex-gaussian-exponential}. The result follows from
Corollary~\ref{cor:point-process-fourier-clt} and the continuous mapping
theorem.
\end{proof}

\subsection{Averages over \texorpdfstring{$L^2$}{L2}-orthonormal functions}
\label{subsec:orthonormal-test-functions-scattering}

Let $h_1,\ldots,h_q\in C_c^\infty(\bR^d;\bC)$ be orthonormal in
$L^2(\bR^d)$. For $\xi\in E_\eta$, define
\begin{equation}
\label{eq:orthonormal-spectral-average}
    Q_R(\xi):=\frac1q\sum_{j=1}^q|Z_R(\xi;h_j)|^2.
\end{equation}
Since
$\int h_a\overline{h_b}\,d\lambda_d=\delta_{ab}$,
Theorem~\ref{thm:smooth-continuum-fourier-clt} gives
\[
    \bigl(Z_R(\xi;h_1),\ldots,Z_R(\xi;h_q)\bigr)
    \dto (G_1,\ldots,G_q),
\]
where $G_1,\ldots,G_q$ are independent
$\mathcal{CN}(0,s_\eta(\xi))$ variables. Therefore
$|G_j|^2/s_\eta(\xi)$ are independent $\operatorname{Exp}(1)$ variables when
$s_\eta(\xi)>0$, and their sum has the $\Gamma(q,1)$ distribution. The same
conclusion at $s_\eta(\xi)=0$ is understood in the degenerate sense.

\begin{corollary}[Gamma limit]
\label{cor:orthonormal-gamma-limit}
For every $\xi\in E_\eta$,
\begin{equation}
\label{eq:orthonormal-gamma-limit}
    Q_R(\xi)
    \dto \frac{s_\eta(\xi)}{q}\,\Gamma(q,1).
\end{equation}
Here $\Gamma(q,1)$ has shape $q$ and scale $1$; its mean and variance are both
$q$. If $\eta$ is a point process of positive intensity, then
\begin{equation}
\label{eq:orthonormal-structure-factor-gamma}
    \frac{Q_R(\xi)}{\rho_\eta}
    \dto \frac{S_\eta(\xi)}{q}\,\Gamma(q,1).
\end{equation}
\end{corollary}

The limiting gamma law in \eqref{eq:orthonormal-gamma-limit} has mean
$s_\eta(\xi)$ and variance $s_\eta(\xi)^2/q$. At frequencies with
$s_\eta(\xi)>0$, its relative variance is $1/q$.
 \section{A zero-entropy counterexample}
\label{sec:counterexample}

We show that regularity of the Bartlett spectrum is not sufficient for a
Fourier central limit theorem. We construct a stationary
ergodic random measure with zero entropy and a bounded continuous Bartlett
density that is positive almost everywhere, while its normalized Fourier
transforms remain uniformly bounded along a sequence of expanding intervals.
The construction starts from the Rudin--Shapiro sequence. Its two-point
correlations are the same as those of a fair Bernoulli $\{\pm1\}$-sequence,
but its finite exponential sums satisfy a deterministic $O(\sqrt N)$ bound.

\subsection{The Rudin--Shapiro system}
\label{subsec:rudin-shapiro-sums}

For $n\geq0$, write the binary expansion of $n$ as
\[
    n=\sum_{j\geq0} e_j(n)2^j,
    \qquad e_j(n)\in\{0,1\},
\]
and let
\[
    b_n:=\sum_{j\geq0} e_j(n)e_{j+1}(n),
    \qquad
    a_n:=(-1)^{b_n}.
\]
Thus $b_n$ counts the occurrences of the block $11$ in the binary expansion
of $n$. The sequence $(a_n)_{n\geq0}$ begins
\[
    1,1,1,-1,1,1,-1,1,1,1,1,-1,-1,-1,1,-1,\ldots
\]
and is the one-sided Rudin--Shapiro sequence. The sequence goes back to work
of Golay, Shapiro and Rudin \cite{Gol51,Sha51,Rud59}; see also
\cite{AS03,Que10,BG13,Maz24} for its automatic, substitution, dynamical and
spectral descriptions.

A convenient substitution model uses the alphabet $\{0,1,2,3\}$ and the
constant-length substitution
\begin{equation}
\label{eq:rs-substitution}
    0\mapsto02,\qquad
    1\mapsto32,\qquad
    2\mapsto01,\qquad
    3\mapsto31.
\end{equation}
Starting from $0$ gives a one-sided fixed point. Under the coding
\[
    0,2\mapsto 1,
    \qquad
    1,3\mapsto -1,
\]
its image is $(a_n)_{n\geq0}$; see \cite[Section~2]{Maz24}. Let
$\mathcal L_{\rm RS}$ be the set of finite words occurring in
$(a_n)_{n\geq0}$, and define
\begin{equation}
\label{eq:rs-hull}
    Y:=\bigl\{y\in\{-1,1\}^{\bZ}:
      \text{every finite subword of $y$ belongs to $\mathcal L_{\rm RS}$}\bigr\}.
\end{equation}
This is the image, under the binary coding above, of the two-sided subshift
associated with \eqref{eq:rs-substitution}. The substitution is primitive, so
its two-sided subshift is minimal and uniquely ergodic; see, for example,
\cite{Que10}. The factor $Y$ is therefore also minimal and uniquely ergodic.
Let $T:Y\to Y$ be the left shift and let $\nu$ be its unique invariant
probability measure. Finally, set
\[
    g(y):=y_0.
\]
The substitution gives each of its four letters frequency $1/4$; after the binary coding, the two symbols therefore have equal $\nu$-frequency, so $\int_Y g\,d\nu=0$. The
Rudin--Shapiro autocorrelation is
\begin{equation}
\label{eq:rs-autocorrelation}
    \int_Y g(y)g(T^m y)\,d\nu(y)=\delta_{m,0},
    \qquad m\in\bZ;
\end{equation}
see \cite[Lemma~4.1]{Maz24} and \cite[Section~10.2]{BG13}.

For $|z|=1$, put
\begin{equation}
\label{eq:rs-exponential-sum}
    P_N(y,z):=\sum_{j=0}^{N-1}g(T^j y)z^j.
\end{equation}
The autocorrelation identity gives the exact $L^2$ size of these sums, while
a theorem of Balister gives a uniform pointwise bound.

\begin{lemma}[Rudin--Shapiro bounds]
\label{lem:rs-bounds}
For every $N\geq1$ and $|z|=1$,
\begin{equation}
\label{eq:rs-L2-and-uniform}
    \int_Y |P_N(y,z)|^2\,d\nu(y)=N,
    \qquad
    \sup_{y\in Y}|P_N(y,z)|\leq\sqrt{10N}.
\end{equation}
\end{lemma}

\begin{proof}
Expanding the square in \eqref{eq:rs-exponential-sum} and using
\eqref{eq:rs-autocorrelation} gives
\[
    \int_Y |P_N(y,z)|^2\,d\nu(y)
    =\sum_{j,k=0}^{N-1}z^j\overline{z}^k
      \int_Y g(T^j y)g(T^k y)\,d\nu(y)
    =N.
\]

Balister proves that, for $m\geq0$, $N\geq1$, and $|z|=1$,
\begin{equation}
\label{eq:balister-interval-bound}
    \left|\sum_{j=m}^{m+N-1}a_jz^j\right|\leq\sqrt{10N};
\end{equation}
see \cite[Theorem~3]{Bal19}. By the definition of $Y$ in
\eqref{eq:rs-hull}, for every $y\in Y$ the word
$(y_0,\ldots,y_{N-1})$ occurs in $(a_n)_{n\geq0}$. Hence there is $m\geq0$
such that $g(T^j y)=a_{m+j}$ for $0\leq j<N$. Therefore
\[
    P_N(y,z)
    =z^{-m}\sum_{j=m}^{m+N-1}a_jz^j,
\]
and \eqref{eq:balister-interval-bound} gives the second estimate in
\eqref{eq:rs-L2-and-uniform}.
\end{proof}

Since \eqref{eq:rs-substitution} is a primitive constant-length substitution,
its subshift has zero topological entropy; see \cite{Que10}. The binary
Rudin--Shapiro system is a factor of this subshift, so
\begin{equation}
\label{eq:rs-zero-entropy}
    h_\nu(T)=0.
\end{equation}

\subsection{The suspension and its Bartlett spectrum}
\label{subsec:rs-suspension}

We next pass from the discrete system to a stationary random measure on
$\bR$. Fix $\alpha>0$ and use the half-open constant-roof model
\[
    X:=Y\times[0,\alpha),
    \qquad
    \mu:=\nu\otimes\alpha^{-1}\lambda_1|_{[0,\alpha)}.
\]
For $x=(y,s)\in X$ and $t\in\bR$, let
\[
    n=n(t,s):=\left\lfloor\frac{s+t}{\alpha}\right\rfloor,
    \qquad
    \phi_t(y,s):=\bigl(T^n y,s+t-n\alpha\bigr).
\]
This is the constant-roof suspension flow. Define $F(y,s):=g(y)$. Then
$\int_XF\,d\mu=0$.

Choose $\rho>0$ and $0<\eps<\rho$. For $x\in X$, define
\begin{equation}
\label{eq:rs-random-measure}
    \omega_x(dt):=\bigl(\rho+\eps F(\phi_t x)\bigr)\,dt,
\end{equation}
and let $\eta$ be the law of $\omega_x$ under $\mu$. For every
$f\in C_c(\bR)$, the map
$x\mapsto\omega_x(f)=\int f(t)(\rho+\eps F(\phi_t x))\,dt$ is measurable by
Fubini's theorem; hence $x\mapsto\omega_x$ is Borel for the vague topology.
Since $|F|=1$ and $\eps<\rho$, $\omega_x$ is a positive Radon measure.
Moreover, for every
$t\in\bR$,
\[
    \omega_{\phi_t x}=t.\omega_x
\]
under \eqref{eq:translation-action-random-measures}. Thus $\eta$ is
stationary. It is ergodic because it is a factor of the ergodic suspension.
For every bounded Borel set $A$,
\[
    \omega_x(A)\leq(\rho+\eps)\lambda_1(A),
\]
so $\eta$ has local second moments, and its intensity is $\rho$.

For $\bR$-flows the spatial entropy used in
Subsection~\ref{subsec:entropy-cpe} is the usual entropy per unit time.
Hence, by \eqref{eq:rs-zero-entropy} and Abramov's formula \cite{Abr59}, the
constant-roof suspension has entropy $\alpha^{-1}h_\nu(T)=0$. Its
random-measure factor therefore has zero entropy as well.

The centered random measure is
\begin{equation}
\label{eq:rs-centered-measure}
    M_x(dt)=\eps F(\phi_t x)\,dt.
\end{equation}
We now compute the covariance of $F$. Write $t=m\alpha+u$ with $m\geq0$ and
$0\leq u<\alpha$. If $x=(y,s)$, then
\[
    F(\phi_t x)
    =
    \begin{cases}
        g(T^m y),&0\leq s<\alpha-u,\\
        g(T^{m+1}y),&\alpha-u\leq s<\alpha.
    \end{cases}
\]
Averaging first over $s$ and then over $y$ gives
\begin{equation}
\label{eq:rs-suspension-covariance-step}
    \E_\mu\bigl[F(\phi_{m\alpha+u}x)F(x)\bigr]
    =\left(1-\frac{u}{\alpha}\right)
      \int_Y g(T^m y)g(y)\,d\nu(y)
     +\frac{u}{\alpha}
      \int_Y g(T^{m+1} y)g(y)\,d\nu(y).
\end{equation}
Using \eqref{eq:rs-autocorrelation} and the symmetry
$\E[F(\phi_{-t}x)F(x)]=\E[F(\phi_t x)F(x)]$, we obtain
\begin{equation}
\label{eq:rs-triangular-covariance}
    R_F(t):=\E_\mu[F(\phi_t x)F(x)]
    =\left(1-\frac{|t|}{\alpha}\right)_+.
\end{equation}

Set
\begin{equation}
\label{eq:rs-c-alpha}
    c_\alpha(\xi):=\int_0^\alpha\chi_\xi(t)\,dt.
\end{equation}
Since
\[
    R_F=\alpha^{-1}1_{[0,\alpha]}*1_{[-\alpha,0]},
\]
its Fourier transform is $\widehat R_F=\alpha^{-1}|c_\alpha|^2$. For
$f,g\in\mathcal S(\bR)$, \eqref{eq:rs-centered-measure}, stationarity, and
Fubini give
\begin{align*}
 \Cov_\eta\bigl(M(f),M(g)\bigr)
 &=\eps^2\iint_{\bR^2} f(t)\overline{g(u)}R_F(t-u)\,dt\,du\\
 &=\eps^2\int_\bR \hat f(\xi)\overline{\hat g(\xi)}
       \widehat R_F(\xi)\,d\xi.
\end{align*}
Comparison with the Bartlett covariance identity therefore shows that the
Bartlett spectrum of $\eta$ is absolutely continuous with density
\begin{equation}
\label{eq:rs-bartlett-density}
    s_\eta(\xi)
    =\frac{\eps^2}{\alpha}|c_\alpha(\xi)|^2
    =\eps^2\alpha
      \left(\frac{\sin(\pi\alpha\xi)}{\pi\alpha\xi}\right)^2,
\end{equation}
where the value at $\xi=0$ is given by continuity. Hence $s_\eta$ is bounded
and continuous, and
\begin{equation}
\label{eq:rs-bartlett-positive-set}
    s_\eta(\xi)>0
    \quad\Longleftrightarrow\quad
    \xi\notin\alpha^{-1}(\bZ\setminus\{0\}).
\end{equation}
In particular, $s_\eta$ is positive for $\lambda_1$-almost every $\xi$.

\subsection{Failure of the Fourier limit}
\label{subsec:rs-fourier-failure}

The covariance computation above determines the limiting second moments. We
now use Lemma~\ref{lem:rs-bounds} to show that the corresponding Fourier
transforms cannot have Gaussian limits.

For $N\geq1$ and $\xi\in\bR$, define
\begin{equation}
\label{eq:rs-suspension-fourier-integral}
    I_N(x,\xi)
    :=\int_0^{N\alpha}\chi_\xi(t)F(\phi_t x)\,dt.
\end{equation}
Then
\begin{equation}
\label{eq:rs-normalized-centered-transform}
    W_N(\xi)
    :=\frac{1}{\sqrt{N\alpha}}
      M_x(1_{[0,N\alpha]}\chi_\xi)
    =\frac{\eps}{\sqrt{N\alpha}}I_N(x,\xi).
\end{equation}
We suppress the dependence of $W_N(\xi)$ on $x\in X$.

Write $x=(y,s)$, with $0\leq s<\alpha$, and define
$G_y(u)=g(T^j y)$ for $j\alpha\leq u<(j+1)\alpha$. Then
$F(\phi_t x)=G_y(t+s)$, so
\[
    I_N(x,\xi)
    =\overline{\chi_\xi(s)}
      \int_s^{N\alpha+s}\chi_\xi(u)G_y(u)\,du.
\]
On the $j$th complete roof interval,
\[
    \int_{j\alpha}^{(j+1)\alpha}
       \chi_\xi(u)G_y(u)\,du
    =c_\alpha(\xi)g(T^j y)e^{2\pi i\alpha j\xi}.
\]
The intervals $[s,N\alpha+s]$ and $[0,N\alpha]$ differ by two intervals of
total length at most $2\alpha$. Since $|G_y|=1$, their contribution has
absolute value at most $2\alpha$. Therefore
\begin{equation}
\label{eq:rs-suspension-reduction}
    I_N(x,\xi)
    =\overline{\chi_\xi(s)}
      \left(c_\alpha(\xi)
      P_N\bigl(y,e^{2\pi i\alpha\xi}\bigr)+E_N(x,\xi)\right),
    \qquad |E_N(x,\xi)|\leq2\alpha.
\end{equation}
Lemma~\ref{lem:rs-bounds} gives
\begin{equation}
\label{eq:rs-uniform-normalized-bound}
    |W_N(\xi)|
    \leq\eps\sqrt{\frac{10}{\alpha}}\,|c_\alpha(\xi)|
      +2\eps\sqrt{\frac{\alpha}{N}}.
\end{equation}
Thus, for each fixed $\xi$, all laws of $W_N(\xi)$ are supported in one
bounded disk, independently of $N$.

On the other hand, stationarity and \eqref{eq:rs-triangular-covariance} give
\begin{align}
\label{eq:rs-second-moment-limit}
    \E_\mu|W_N(\xi)|^2
    &=\eps^2\int_{-\alpha}^{\alpha}
       \left(1-\frac{|t|}{N\alpha}\right)
       \chi_\xi(t)R_F(t)\,dt \\
    &\longrightarrow
      \eps^2\int_{-\alpha}^{\alpha}\chi_\xi(t)R_F(t)\,dt
      =s_\eta(\xi).
\end{align}
The second moments therefore have exactly the asymptotic value prescribed by
the Bartlett density, even though the laws are confined to a fixed bounded
set.

\begin{theorem}[Regular Bartlett spectrum without a Fourier CLT]
\label{thm:rs-counterexample}
The stationary ergodic random measure $\eta$ defined by
\eqref{eq:rs-random-measure} has zero entropy and Bartlett density
\eqref{eq:rs-bartlett-density}. For every $\xi$ with $s_\eta(\xi)>0$, the sequence $W_N(\xi)$ does not
converge to any centered complex Gaussian law. In particular,
\begin{equation}
\label{eq:rs-no-gaussian-limit}
    W_N(\xi)\not\dto\mathcal{CN}(0,s_\eta(\xi)),
\end{equation}
and
\begin{equation}
\label{eq:rs-no-exponential-limit}
    |W_N(\xi)|^2
    \not\dto s_\eta(\xi)\operatorname{Exp}(1).
\end{equation}
Moreover, for every such $\xi$ there is $h\in C_c^\infty(\bR;\bC)$ for
which $R^{-1/2}M(h(\cdot/R)\chi_\xi)$ does not converge, as $R\to\infty$,
to $\mathcal{CN}(0,s_\eta(\xi)\lVert h\rVert_2^2)$. In particular, a bounded
continuous Bartlett density, positive $\lambda_1$-almost everywhere, does not
imply an almost-everywhere-frequency empirical Fourier CLT.
\end{theorem}

\begin{proof}
Fix $\xi$ with $s_\eta(\xi)>0$. By
\eqref{eq:rs-uniform-normalized-bound}, there is a closed disk $K\subset\bC$
containing the support of the law of $W_N(\xi)$ for every $N$. Any weak limit
of these laws is therefore supported in $K$. Every nondegenerate centered
complex Gaussian law has unbounded support, so the only centered Gaussian law
that could be supported in $K$ is the degenerate law $\delta_0$.

That possibility is also excluded. The variables $W_N(\xi)$ are centered and
uniformly bounded, and \eqref{eq:rs-second-moment-limit} gives
\[
    \E_\mu|W_N(\xi)|^2\longrightarrow s_\eta(\xi)>0.
\]
If $W_N(\xi)\dto\delta_0$, uniform boundedness would force the second moments
to converge to $0$, a contradiction. Thus $W_N(\xi)$ has no centered Gaussian
weak limit at all; in particular \eqref{eq:rs-no-gaussian-limit} holds. The
same support bound for $|W_N(\xi)|^2$ excludes the unbounded law
$s_\eta(\xi)\operatorname{Exp}(1)$ in \eqref{eq:rs-no-exponential-limit}.

It remains to relate this sharp-window obstruction to the smooth-window CLT
of Theorem~\ref{thm:smooth-continuum-fourier-clt}. The interval estimate also holds for arbitrary blocks of the two-sided hull:
for $m\in\bZ$ and $L\geq1$,
\[
 \left|\sum_{j=m}^{m+L-1}g(T^j y)z^j\right|
 =\left|P_L(T^m y,z)\right|
 \leq\sqrt{10L},
 \qquad |z|=1.
\]
Now fix $R\geq1$ and $x=(y,s)$. After the change of variables used in
\eqref{eq:rs-suspension-reduction}, the interval $[s-R,s+R]$ contains one
block of $L\leq 2R/\alpha+2$ complete roof intervals and at most two partial
end intervals, of total length at most $2\alpha$. Hence
\[
 \bigl|M_x(\ind_{B_R}\chi_\xi)\bigr|
 \leq \eps |c_\alpha(\xi)|\sqrt{10L}+2\eps\alpha.
\]
Since $L=O(R)$, this gives, for each fixed $\xi$,
\begin{equation}
\label{eq:rs-symmetric-window-uniform-bound}
    \sup_{R\geq1}\sup_{x\in X}
    R^{-1/2}
    \bigl|M_x(\ind_{B_R}\chi_\xi)\bigr|<\infty.
\end{equation}
Suppose that the conclusion of
Theorem~\ref{thm:smooth-continuum-fourier-clt} held at this $\xi$ for all
smooth compactly supported functions. Since $s_\eta$ is bounded and
continuous, the approximation argument in the proof of
Theorem~\ref{thm:ball-fourier-clt} would then give
\[
    R^{-1/2}M(\ind_{B_R}\chi_\xi)
    \dto\mathcal{CN}(0,2s_\eta(\xi)).
\]
Because $s_\eta(\xi)>0$, this limit is nondegenerate and has unbounded
support, whereas \eqref{eq:rs-symmetric-window-uniform-bound} places every
ball-window law in a single bounded disk. This contradiction shows that the
smooth-window Fourier CLT also fails at every $\xi$ for which
$s_\eta(\xi)>0$.
\end{proof}

Taking $R_N=N\alpha$ proves Theorem~\ref{thm:intro-counterexample}. The
example is a random-measure counterexample. Whether the same fixed-frequency
obstruction can be realized by a stationary point process is left open.

\appendix
\section{A lexicographic Hilbert-space decomposition}
\label{app:lexicographic}

The lemma below is used in the proof of
Proposition~\ref{prop:wandering-innovation-decomposition}. Recall that a cut of the
lexicographically ordered group $\bZ^d$ is a pair $(L,R)$ of nonempty sets such that
$\bZ^d=L\sqcup R$ and $\ell<_{\rm lex}r$ for every $\ell\in L$ and $r\in R$.
A cut is nonprincipal if $L$ has no greatest element and $R$ has no least element.
The lemma shows that, when an increasing family of closed subspaces is continuous at
all nonprincipal cuts, its orthogonal differences are precisely the successive
differences $\mathscr H_k\ominus\mathscr H_{k-e_d}$.

\begin{lemma}[Lexicographic difference decomposition]
\label{lem:lexicographic-jump-decomposition}
Let $(\mathscr H_k)_{k\in\bZ^d}$ be an increasing family of closed subspaces of
a Hilbert space, indexed by the lexicographic order, and put
\[
    \mathscr H_-:=\bigcap_{k\in\bZ^d}\mathscr H_k,
    \qquad
    \mathscr H_+:=\overline{\bigcup_{k\in\bZ^d}\mathscr H_k}.
\]
Suppose that for every nonprincipal cut $\bZ^d=L\sqcup R$,
\[
    \overline{\bigcup_{k\in L}\mathscr H_k}
    =\bigcap_{k\in R}\mathscr H_k.
\]
Then
\begin{equation}
\label{eq:abstract-lex-jump-decomposition}
    \mathscr H_+\ominus\mathscr H_-
    =\bigoplus_{k\in\bZ^d}
      \bigl(\mathscr H_k\ominus\mathscr H_{k-e_d}\bigr).
\end{equation}
\end{lemma}

\begin{proof}
For $d=1$, \eqref{eq:abstract-lex-jump-decomposition} is the orthogonal
decomposition of an increasing sequence into its successive differences.
Assume the lemma in dimension $d-1$ and write
$k=(m,r)$ with $m\in\bZ$ and $r\in\bZ^{d-1}$. Set
\[
    \mathscr H_m^-:=\bigcap_{r\in\bZ^{d-1}}\mathscr H_{(m,r)},
    \qquad
    \mathscr H_m^+:=\overline{\bigcup_{r\in\bZ^{d-1}}\mathscr H_{(m,r)}}.
\]
The nonprincipal cut between the blocks with first coordinate $m-1$ and $m$
has left endpoint $\mathscr H_{m-1}^+$ and right endpoint
$\mathscr H_m^-$, hence
$\mathscr H_{m-1}^+=\mathscr H_m^-$. To apply the induction hypothesis inside
the $m$-th block, let $(L',R')$ be a nonprincipal cut of $\bZ^{d-1}$ and
extend it to a cut $(L,R)$ of $\bZ^d$ by adjoining all blocks with first
coordinate $<m$ to $L$ and all blocks with first coordinate $>m$ to $R$.
Then
\[
 \overline{\bigcup_{r\in L'}\mathscr H_{(m,r)}}
 =\overline{\bigcup_{k\in L}\mathscr H_k},
 \qquad
 \bigcap_{r\in R'}\mathscr H_{(m,r)}
 =\bigcap_{k\in R}\mathscr H_k.
\]
Indeed, earlier blocks are contained in every space from the $m$-th block,
while later blocks contain every such space. Thus the cut continuity in
$\bZ^d$ gives the required cut continuity inside the $m$-th block, and the
induction hypothesis gives
\[
    \mathscr H_m^+\ominus\mathscr H_m^-
    =\bigoplus_{r\in\bZ^{d-1}}
      \bigl(\mathscr H_{(m,r)}\ominus
             \mathscr H_{(m,r)-e_d}\bigr).
\]
The spaces on the left are the successive differences of the increasing
sequence $(\mathscr H_m^+)_{m\in\bZ}$. Since
\[
    \bigcap_m\mathscr H_m^+
    =\bigcap_m\mathscr H_{m+1}^-
    =\bigcap_{m,r}\mathscr H_{(m,r)}
    =\mathscr H_-,
\]
and
\[
    \overline{\bigcup_m\mathscr H_m^+}
    =\overline{\bigcup_{m,r}\mathscr H_{(m,r)}}
    =\mathscr H_+,
\]
decomposing this bilateral sequence and then substituting the block
decompositions gives
\eqref{eq:abstract-lex-jump-decomposition}.
\end{proof}

\section*{Statements and Declarations}

\paragraph{\textbf{Funding.}}
This work was supported by the Swedish Research Council under grant
VR 11253322.

\paragraph{\textbf{Competing interests.}}
The author has no relevant financial or non-financial interests to disclose.

\paragraph{\textbf{Data availability.}}
No datasets were generated or analysed during the current study.


\begin{thebibliography}{HGBL23}

\bibitem[Abr59]{Abr59}
Abramov, L.~M.
\newblock \emph{On the entropy of a flow}.
\newblock Dokl. Akad. Nauk SSSR 128 (1959), 873--875.

\bibitem[AS03]{AS03}
Allouche, J.-P. and Shallit, J.
\newblock \emph{Automatic Sequences: Theory, Applications, Generalizations}.
\newblock Cambridge University Press, Cambridge, 2003.

\bibitem[Aub00]{Aub00}
Aubin, J.-P.
\newblock \emph{Applied Functional Analysis}.
\newblock Second edition, Wiley, New York, 2000.

\bibitem[BG13]{BG13}
Baake, M. and Grimm, U.
\newblock \emph{Aperiodic Order. Volume 1: A Mathematical Invitation}.
\newblock Cambridge University Press, Cambridge, 2013.

\bibitem[Bal19]{Bal19}
Balister, P.
\newblock \emph{Bounds on Rudin--Shapiro polynomials of arbitrary degree}.
\newblock Preprint, arXiv:1909.08777, 2019.

\bibitem[Bar63]{Bar63}
Bartlett, M.~S.
\newblock \emph{The spectral analysis of point processes}.
\newblock J. Roy. Statist. Soc. Ser. B 25 (1963), 264--281.

\bibitem[Bil99]{Bil99}
Billingsley, P.
\newblock \emph{Convergence of Probability Measures}.
\newblock Second edition, Wiley, New York, 1999.

\bibitem[Bjo26a]{Bjo26a}
Bj\"orklund, M.
\newblock \emph{Stealthy point processes and lattice induction}.
\newblock Preprint, arXiv:2607.25616, 2026.

\bibitem[Bjo26b]{Bjo26b}
Bj\"orklund, M.
\newblock \emph{Hyperuniform Delone realizations and rigidity}.
\newblock Preprint, arXiv:2608.16547, 2026.

\bibitem[BG20]{BG20}
Bj\"orklund, M. and Gorodnik, A.
\newblock \emph{Central limit theorems for group actions which are exponentially mixing of all orders}.
\newblock J. Anal. Math. 141 (2020), 457--482.

\bibitem[BYY26]{BYY26}
B{\l}aszczyszyn, B., Yogeshwaran, D., and Yukich, J.~E.
\newblock \emph{Limit theory for Lipschitz-localized statistics in random geometric models}.
\newblock Preprint, arXiv:2605.28430, 2026.

\bibitem[Bol82]{Bol82}
Bolthausen, E.
\newblock \emph{On the central limit theorem for stationary mixing random fields}.
\newblock Ann. Probab. 10 (1982), no.~4, 1047--1050.

\bibitem[Bri72]{Bri72}
Brillinger, D.~R.
\newblock \emph{The spectral analysis of stationary interval functions}.
\newblock In: Proc. Sixth Berkeley Symp. Math. Statist. Probab., Vol.~I,
University of California Press, Berkeley, 1972, 483--513.

\bibitem[Bri82]{Bri82}
Brillinger, D.~R.
\newblock \emph{Asymptotic normality of finite Fourier transforms of stationary generalized processes}.
\newblock J. Multivariate Anal. 12 (1982), 64--71.

\bibitem[BD87]{BD87}
Burton, R. and Denker, M.
\newblock \emph{On the central limit theorem for dynamical systems}.
\newblock Trans. Amer. Math. Soc. 302 (1987), no.~2, 715--726.

\bibitem[CC13]{CC13}
Cohen, G. and Conze, J.-P.
\newblock \emph{The CLT for rotated ergodic sums and related processes}.
\newblock Discrete Contin. Dyn. Syst. 33 (2013), no.~9, 3981--4002.

\bibitem[DVJ03]{DVJ03}
Daley, D.~J. and Vere-Jones, D.
\newblock \emph{An Introduction to the Theory of Point Processes. Volume I: Elementary Theory and Methods}.
\newblock Second edition, Springer, New York, 2003.

\bibitem[DG12]{DG12}
Dooley, A.~H. and Golodets, V.~Ya.
\newblock \emph{On the entropy of actions of nilpotent Lie groups and their lattice subgroups}.
\newblock Ergodic Theory Dynam. Systems 32 (2012), no.~2, 535--573.

\bibitem[FRS07]{FRS07}
Fa\"y, G., Roueff, F., and Soulier, P.
\newblock \emph{Estimation of the memory parameter of the infinite-source Poisson process}.
\newblock Bernoulli 13 (2007), no.~2, 473--491.

\bibitem[Gol51]{Gol51}
Golay, M.~J.~E.
\newblock \emph{Static multislit spectrometry and its application to the panoramic display of infrared spectra}.
\newblock J. Opt. Soc. Amer. 41 (1951), 468--472.

\bibitem[HGBL23]{HGBL23}
Hawat, D., Gautier, G., Bardenet, R., and Lachi\`eze-Rey, R.
\newblock \emph{On estimating the structure factor of a point process, with applications to hyperuniformity}.
\newblock Stat. Comput. 33 (2023), article 61.

\bibitem[HO74]{HO74}
Hawkes, A.~G. and Oakes, D.
\newblock \emph{A cluster process representation of a self-exciting process}.
\newblock J. Appl. Probab. 11 (1974), 493--503.


\bibitem[Kal82]{Kal82}
Kalikow, S.~A.
\newblock \emph{$T,T^{-1}$ transformation is not loosely Bernoulli}.
\newblock Ann. of Math. (2) 115 (1982), no.~2, 393--409.

\bibitem[Kam81]{Kam81}
Kami\'nski, B.
\newblock \emph{The theory of invariant partitions for $\bZ^d$-actions}.
\newblock Bull. Acad. Polon. Sci. S\'er. Sci. Math. 29 (1981), 349--362.

\bibitem[Kam91]{Kam91}
Kami\'nski, B.
\newblock \emph{Decreasing nets of $\sigma$-algebras and their applications to ergodic theory}.
\newblock Tohoku Math. J. (2) 43 (1991), 263--274.

\bibitem[Kam96]{Kam96}
Kami\'nski, B.
\newblock \emph{Invariant $\sigma$-algebras for $\bZ^d$-actions and their applications}.
\newblock In: M. Pollicott and K. Schmidt (eds.), \emph{Ergodic Theory of
$\bZ^d$ Actions}, London Math. Soc. Lecture Note Ser. 228, Cambridge
University Press, Cambridge, 1996, 403--414.

\bibitem[KL94]{KL94}
Kami\'nski, B. and Liardet, P.
\newblock \emph{Spectrum of multidimensional dynamical systems with positive entropy}.
\newblock Studia Math. 108 (1994), 77--85.

\bibitem[KLH26]{KLH26}
Klatt, M.~A., Last, G., and Henze, N.
\newblock \emph{A genuine test for hyperuniformity}.
\newblock Preprint, arXiv:2210.12790, version 2, 2026.

\bibitem[KV22]{KV22}
Kosloff, Z. and Voln\'y, D.
\newblock \emph{Local limit theorem in deterministic systems}.
\newblock Ann. Inst. Henri Poincar\'e Probab. Stat. 58 (2022), no.~1, 548--566.

\bibitem[KY24]{KY24}
Krishnapur, M. and Yogeshwaran, D.
\newblock \emph{Stationary random measures: covariance asymptotics, variance bounds and central limit theorems}.
\newblock Preprint, arXiv:2411.08848, 2024.

\bibitem[Mas26]{Mas26}
Mastrilli, G.
\newblock \emph{Asymptotic fluctuations of smooth linear statistics of independently perturbed lattices}.
\newblock J. Stat. Phys. 193 (2026), article 24.

\bibitem[MBL24]{MBL24}
Mastrilli, G., B\l aszczyszyn, B., and Lavancier, F.
\newblock \emph{Estimating the hyperuniformity exponent of point processes}.
\newblock Preprint, arXiv:2407.16797, 2024.

\bibitem[Maz24]{Maz24}
Maz\'a\v{c}, J.
\newblock \emph{Correlation functions of the Rudin--Shapiro sequence}.
\newblock Indag. Math. (N.S.) 35 (2024), no.~5, 771--795.

\bibitem[McL74]{McL74}
McLeish, D.~L.
\newblock \emph{Dependent central limit theorems and invariance principles}.
\newblock Ann. Probab. 2 (1974), no.~4, 620--628.

\bibitem[Mei74]{Mei74}
Meilijson, I.
\newblock \emph{Mixing properties of a class of skew-products}.
\newblock Israel J. Math. 19 (1974), 266--270.

\bibitem[Osa21]{Osa21}
Osada, S.
\newblock \emph{Isomorphisms between determinantal point processes with translation-invariant kernels and Poisson point processes}.
\newblock Ergodic Theory Dynam. Systems 41 (2021), no.~12, 3807--3820.

\bibitem[PW10]{PW10}
Peligrad, M. and Wu, W.~B.
\newblock \emph{Central limit theorem for Fourier transforms of stationary processes}.
\newblock Ann. Probab. 38 (2010), no.~5, 2009--2022.

\bibitem[PZ19]{PZ19}
Peligrad, M. and Zhang, N.
\newblock \emph{Central limit theorem for Fourier transform and periodogram of random fields}.
\newblock Bernoulli 25 (2019), no.~1, 499--520.

\bibitem[Que10]{Que10}
Queff\'elec, M.
\newblock \emph{Substitution Dynamical Systems---Spectral Analysis}.
\newblock Second edition, Lecture Notes in Mathematics 1294, Springer, Berlin, 2010.

\bibitem[RS61]{RS61}
Rohlin, V.~A. and Sinai, Ya.~G.
\newblock \emph{Construction and properties of invariant measurable partitions}.
\newblock Dokl. Akad. Nauk SSSR 141 (1961), no.~5, 1038--1041.

\bibitem[Rud59]{Rud59}
Rudin, W.
\newblock \emph{Some theorems on Fourier coefficients}.
\newblock Proc. Amer. Math. Soc. 10 (1959), 855--859.

\bibitem[Sha51]{Sha51}
Shapiro, H.~S.
\newblock \emph{Extremal problems for polynomials and power series}.
\newblock M.S. thesis, Massachusetts Institute of Technology, 1951.

\bibitem[Smo75]{Smo75}
Smorodinsky, M.
\newblock \emph{Construction of $K$-flows}.
\newblock Adv. Math. 15 (1975), 207--215.

\bibitem[TBvB13]{TBvB13}
Teichmann, J., Ballani, F., and van den Boogaart, K.~G.
\newblock \emph{Generalizations of Mat\'ern's hard-core point processes}.
\newblock Spatial Statist. 3 (2013), 33--53.

\bibitem[Vol99]{Vol99}
Voln\'y, D.
\newblock \emph{Invariance principles and Gaussian approximation for strictly stationary processes}.
\newblock Trans. Amer. Math. Soc. 351 (1999), no.~8, 3351--3371.

\bibitem[Wie39]{Wie39}
Wiener, N.
\newblock \emph{The ergodic theorem}.
\newblock Duke Math. J. 5 (1939), no.~1, 1--18.

\end{thebibliography}
\end{document}